\documentclass[preprint,11pt]{elsarticle}

\usepackage[utf8]{inputenc}

\usepackage{geometry}
\usepackage{amsmath}
\usepackage{amssymb}
\usepackage{amsthm}
\usepackage{enumerate}
\usepackage{tikz}
\usepackage{graphicx}
\usepackage{xcolor}

\newtheorem{teo}{Theorem}[section]
\newtheorem{remark}[teo]{Remark}
\newtheorem{lemma}[teo]{Lemma}
\newtheorem{prop}[teo]{Proposition}
\newtheorem{cor}[teo]{Corollary}
\newtheorem{definition}[teo]{Definition}

\usepackage{hyperref}
\hypersetup{hidelinks,
  pdftitle={Crossover asymptotics and a sharp confinement rate for the viscous Burgers equation},
  pdfauthor={Maicon Sonego and Enrique Zuazua},
  pdfkeywords={viscous Burgers equation, crossover asymptotics, Hopf--Cole transformation, conservative boundary conditions, sharp confinement rate, variable diffusion}
}

\journal{Journal of Differential Equations}

\begin{document}

\begin{frontmatter}

\title{Crossover asymptotics and a sharp confinement rate for the
viscous Burgers equation}

\author[itajuba]{Maicon Sonego\corref{cor1}}
 \ead{mcn.sonego@unifei.edu.br}  

\author[fau,deusto,uam]{Enrique Zuazua}
\ead{enrique.zuazua@fau.de}
\cortext[cor1]{Corresponding author.}

\affiliation[itajuba]{organization={Instituto de Matem\'atica e
    Computa\c{c}\~ao, Universidade Federal de Itajub\'a},
  city={Itajub\'a-MG},
  country={Brazil}}

\affiliation[fau]{organization={Chair for Dynamics, Control, Machine
    Learning and Numerics (Alexander von Humboldt Professorship),
    Department of Mathematics,
    Friedrich--Alexander-Universit\"at Erlangen--N\"urnberg},
  addressline={Cauerstr.~11},
  postcode={91058},
  city={Erlangen},
  country={Germany}}

\affiliation[deusto]{organization={Chair of Computational Mathematics,
    University of Deusto},
  postcode={48007},
  city={Bilbao, Basque Country},
  country={Spain}}

\affiliation[uam]{organization={Departamento de Matem\'aticas,
    Universidad Aut\'onoma de Madrid},
  postcode={28049},
  city={Madrid},
  country={Spain}}

\begin{abstract}
We study the one-dimensional viscous Burgers equation on $(0,L)$ with
the conservative boundary conditions $u_x=-u^2$.  On the half-line, solutions
approach a nonlinear self-similar profile, whereas on a bounded interval they
converge to a nonconstant equilibrium.  We describe explicitly how the
dynamics passes between these states at the critical diffusive scale
$t=cL^2$.  For compactly supported initial data of mass $M$, the rescaled
solution converges as $L\to\infty$ to an explicit crossover profile $\Phi_c$.
In similarity variables, $\Phi_c$ converges to the half-line profile $f_M$ as
$c\downarrow0$; after rescaling to domain variables, it converges to the
interval equilibrium as $c\to\infty$.  We also determine the sharp onset of
confinement:
\[
  \lim_{c\downarrow0}-c\log|\Phi_c(0)-f_M(0)|=1\qquad(M\ne0).
\]
Moreover, on compact sets in similarity variables, the interval and
half-line solutions differ in $C^k$ by at most
$C_{k,\varepsilon}\exp(-(1-\varepsilon)L^2/t)$, uniformly in $L$, and the
exponential constant is optimal.  Thus we identify not only the transition
scale $L^2$, but also the profile governing the crossover and the sharp rate
at which the remote boundary becomes visible.  We finally discuss the
conclusions that persist for space-dependent diffusivity and illustrate the
three asymptotic regimes numerically.
\end{abstract}


\end{frontmatter}

\section{Introduction}

A fundamental feature of diffusive equations is that the large-time behavior of solutions depends crucially on the geometry of the underlying domain. In unbounded domains, solutions typically spread and converge, after suitable rescaling, to self-similar profiles that are universal up to conserved quantities. In bounded domains, confinement eventually takes over and drives the solution toward a stationary state. When the domain is large but finite, these two mechanisms compete. That the diffusive time $t\sim L^{2}$ separates the two regimes is classical. Understanding the nature of the transition at this scale is the central question we address.

A simple and instructive example is the linear heat equation with homogeneous Neumann boundary conditions,
\[
v_t = v_{xx}, \qquad (x,t)\in(0,L)\times(0,\infty), \qquad v_x(0,t)=v_x(L,t)=0,
\]
which conserves mass: $\int_0^L v(x,t)\,dx = M$. On the half-line with Neumann condition at $x=0$, the solution converges, after rescaling, to the Gaussian heat kernel. On the bounded interval, the solution converges exponentially to the spatial average $M/L$. Crucially, before the boundary is felt, the solution on $(0,L)$ is exponentially close to the half-line solution for times $t \ll L^2$; only for $t \gg L^2$ does confinement dominate. This two-scale behavior is a manifestation of the classical \emph{principle of not feeling the boundary}, introduced by Kac \cite{Kac1951}, proved for the heat kernel by Ciesielski \cite{Ciesielski1966}, and later quantified with an explicit rate by van den Berg \cite{vandenBerg1989} (see also Davies \cite{Davies1989}): for the heat equation, the influence of the boundary is exponentially small for times $t \ll \operatorname{dist}(x,\partial\Omega)^2$.

Classical stabilization theory for expanding parabolic domains establishes
qualitatively that a remote boundary is not felt and that the diffusive scale
governs the competition \cite{Ushakov1980,Mukminov1980,Ilin1985,Gushchin1984}.
Here we identify the crossover profile itself, the two nonlinear states it
connects, and the sharp confinement exponent.

In this work, we study the viscous Burgers equation with conservative boundary conditions,
\[
u_t - u_{xx} = (u^2)_x, \qquad (x,t)\in(0,L)\times(0,\infty),
\]
with
\[
u_x(0,t) = -u^2(0,t), \qquad u_x(L,t) = -u^2(L,t).
\]
Written conservatively the equation reads $u_t-\partial_x(u_x+u^2)=0$, so the flux is $-(u_x+u^2)$ and these conditions are precisely the statement that it vanishes at the boundaries; mass is therefore conserved. In the literature, such conditions are sometimes referred to as \emph{no-flux boundary conditions} in the context of the conservation-law formulation \cite{Watanabe2016}; they arise naturally in applications such as the complete separation of fluids in confined geometries. Throughout this paper, we adopt the term \emph{conservative} to highlight the mass-conservation property that is central to our analysis.

The conservative structure provides a common parameter: the total mass $M = \int_0^L u(x,t)\,dx$ is conserved in both the bounded-interval and half-line problems. As recalled in Section~\ref{FI}, on the half-line solutions converge in self-similar variables to a non-Gaussian profile $f_M$ parametrized by $M$, while on the bounded interval they converge exponentially to a non-constant stationary profile $U_M^L$, also parametrized by $M$. The question we address is: how does the bounded-interval dynamics pass from one to the other?

Our contribution has three parts. First, we construct an explicit family
$\{\Phi_c\}_{c>0}$ governing the simultaneous limit $L\to\infty$,
$t=cL^2$. Second, we prove that this family connects the half-line
self-similar state $f_M$ to the finite-interval equilibrium $\mathcal{U}_M$ in their
respective natural variables. Third, we determine the sharp rate at which the
remote boundary becomes visible. More precisely, for compactly supported
initial data, and with $M\ne0$ for the nontrivial lower bound,
\[
\lim_{c\to0^{+}}\,-c\ln\big|\Phi_c(0)-f_M(0)\big| = 1,
\]
and the corresponding comparison with the half-line problem is
\[
\| v^L(\cdot,t) - v^\infty(\cdot,t) \|_{C^k([0,R])} \le C_{k,\varepsilon}\, e^{-(1-\varepsilon)/c},
\qquad t=cL^{2},
\]
on compact sets in similarity variables, uniformly in $L$. The optimal
constant $1$ is the Gaussian cost of the image charge at $x=2L$. Hence the
result resolves the dynamics at the diffusive scale $t\sim L^2$, rather than
merely locating the scale at which confinement begins.

The Hopf--Cole transformation \cite{hopf,cole} maps the equation exactly
to a heat equation with constant Dirichlet data. The boundary-shadowing
mechanism is therefore linear, but the states it connects are nonlinear: $f_M$
is non-Gaussian, $\mathcal{U}_M$ is nonconstant, and $\Phi_c$ deforms one into the
other. Recovering the result for $u=w_x/w$ also requires uniform control of
this quotient and its derivatives on the region $x=O(\sqrt t)$.

Two earlier works address the Burgers equation on a bounded interval and deserve explicit comparison. Kreiss and Kreiss \cite{Kreiss1986} prove convergence to a unique steady state for the initial--boundary value problem and show that the rate of convergence depends on the boundary conditions and may be exponentially slow; Bertini and Ponsiglione \cite{BertiniPonsiglione2012} characterize the stationary solutions variationally, through a Lyapunov functional arising as the large-deviation rate of an associated interacting particle system, and analyze its behavior as the interval length diverges. Both concern the stationary regime $t\gg L^2$ and inhomogeneous Dirichlet data. The conservative boundary conditions considered here produce, after Hopf--Cole, the specific Dirichlet data $w(0,t)=1$, $w(L,t)=e^{M}$, so that the mass $M$ is conserved and parametrizes the stationary profile $U_M^L$; what is new here is not the stationary regime but the diffusive scale $t\sim L^{2}$ and the explicit family linking $f_M$ and $\mathcal U_M$ in their respective natural coordinates.

The distinction from the closest results can be summarized as follows.
The classical ``not feeling the boundary'' literature
\cite{Kac1951,Ciesielski1966,vandenBerg1989,Davies1989} estimates linear heat
kernels or solutions locally before a boundary is reached; it does not identify
the nonlinear profile obtained when $t/L^2$ has a nonzero limit.
Watanabe et al.\ \cite{Watanabe2016} provide well-posedness and
large-time information for the no-flux Burgers problem, while Biler--Karch
\cite{Karch2000} study half-line asymptotics for a Neumann problem on
$(0,\infty)$, in spirit close to Section~\ref{FI}; what is specific to the
conservative condition is the explicit form of the profile $f_M$.  To the best
of our knowledge, none of these results combines
the simultaneous limit $L\to\infty$, $t=cL^2$, the explicit crossover family,
and a matching sharp finite-domain estimate.  These are precisely the contents of Theorem~\ref{thm:intro-main}.

The transient lasts for the natural diffusive time $t\sim L^2$ and is an
instance of Barenblatt's intermediate asymptotics
\cite{BarenblattZeldovich1972,barenblatt1996}.  It is distinct from the
small-viscosity metastability and exponentially slow shock motion studied in
\cite{KT01,BW09,ReynaWard1995,Kreiss1986,MasciaStrani2013}; see
Remark~\ref{rem:notmetastable}.

The following statement gathers the three limits and the uniform
comparison estimate proved separately in Section~\ref{sm}.

\begin{teo}[Main crossover theorem]\label{thm:intro-main}
Let $u_0\in L^1(0,\infty)$ be supported in $[0,R_0]$, with mass $M$, and
let $u^L$ and $u^\infty$ be the corresponding Hopf--Cole solutions on $(0,L)$
and $(0,\infty)$.  Set
\[
 v^L(\xi,t)=\sqrt t\,u^L(\xi\sqrt t,t),\qquad
 v^\infty(\xi,t)=\sqrt t\,u^\infty(\xi\sqrt t,t).
\]
There is an explicit family $\{\Phi_c\}_{c>0}$, defined in
\eqref{eq:Phic}, such that:
\begin{enumerate}
\item for every fixed $c>0$, $k\ge0$, and
$R<c^{-1/2}$,
\[
 \|v^L(\cdot,cL^2)-\Phi_c\|_{C^k([0,R])}=O(L^{-1});
\]
\item for every fixed $R>0$ and $k\ge0$,
$\Phi_c\to f_M$ in $C^k([0,R])$ as $c\downarrow0$, and, if $M\ne0$,
\[
 \lim_{c\downarrow0}-c\log|\Phi_c(0)-f_M(0)|=1;
\]
\item in domain variables,
\[
 \frac1{\sqrt c}\Phi_c\!\left(\frac{y}{\sqrt c}\right)
 \longrightarrow \mathcal U_M(y)
 \quad\text{uniformly for }y\in[0,1]\quad(c\to\infty);
\]
\item for every $R>0$, $k\ge0$, and $\varepsilon\in(0,1)$, there are
$C_{k,\varepsilon},L_0>0$, independent of $c$ and $L$, such that, whenever
$0<c<R^{-2}$, $t=cL^2$, and $L\ge L_0$,
\[
 \|v^L(\cdot,t)-v^\infty(\cdot,t)\|_{C^k([0,R])}
 \le C_{k,\varepsilon} e^{-(1-\varepsilon)/c}.
\]
The exponential constant $1$ is optimal when $M\ne0$.
\end{enumerate}
\end{teo}

Figure~\ref{fig:crossover} illustrates the three windows and the role of $\Phi_c$ in the middle one.

\begin{figure}[htbp]
\centering
\begin{tikzpicture}[xscale=5.05, yscale=0.60, every node/.style={inner sep=1.7pt}]

\definecolor{diffcol}{RGB}{20,130,115}
\definecolor{confcol}{RGB}{195,85,45}

\shade[left color=white, right color=gray!14] (-0.86,-6.20) rectangle (-0.458,0.92);
\shade[left color=gray!14, right color=white] (-0.458,-6.20) rectangle (-0.06,0.92);

\draw[->,gray!65,line width=0.5pt] (-1.45,-6.20) -- (0.52,-6.20);
\draw[->,gray!65,line width=0.5pt] (-1.45,-6.20) -- (-1.45,0.95);
\node[anchor=west,font=\small,gray!85] at (0.55,-6.02) {$c=t/L^{2}$};
\node[gray!85,font=\footnotesize,anchor=south west] at (-1.43,0.98)
      {};

\foreach \x/\lab in {-1.301/{0.05},-1/{0.1},-0.523/{0.3},0/{1},0.301/{2}}{
  \draw[gray!65] (\x,-6.20) -- (\x,-6.34);
  \node[anchor=north,font=\footnotesize,gray!85] at (\x,-6.36) {$\lab$};}
\foreach \y/\lab in {0/{10^{0}},-2/{10^{-2}},-4/{10^{-4}},-6/{10^{-6}}}{
  \draw[gray!65] (-1.45,\y) -- (-1.475,\y);
  \node[anchor=east,font=\footnotesize,gray!85] at (-1.485,\y) {$\lab$};}

\draw[diffcol,line width=1.25pt] plot coordinates {(-1.1739,-6.1940) (-1.1681,-6.1087) (-1.1624,-6.0245) (-1.1566,-5.9415) (-1.1509,-5.8595) (-1.1451,-5.7787) (-1.1394,-5.6988) (-1.1336,-5.6201) (-1.1279,-5.5423) (-1.1221,-5.4656) (-1.1164,-5.3899) (-1.1106,-5.3152) (-1.1049,-5.2415) (-1.0991,-5.1688) (-1.0934,-5.0970) (-1.0876,-5.0261) (-1.0819,-4.9562) (-1.0761,-4.8872) (-1.0704,-4.8191) (-1.0646,-4.7519) (-1.0588,-4.6856) (-1.0531,-4.6202) (-1.0473,-4.5556) (-1.0416,-4.4919) (-1.0358,-4.4290) (-1.0301,-4.3669) (-1.0243,-4.3057) (-1.0186,-4.2453) (-1.0128,-4.1856) (-1.0071,-4.1268) (-1.0013,-4.0687) (-0.9956,-4.0114) (-0.9898,-3.9548) (-0.9841,-3.8990) (-0.9783,-3.8439) (-0.9726,-3.7895) (-0.9668,-3.7359) (-0.9611,-3.6830) (-0.9553,-3.6307) (-0.9496,-3.5792) (-0.9438,-3.5283) (-0.9381,-3.4781) (-0.9323,-3.4285) (-0.9265,-3.3796) (-0.9208,-3.3314) (-0.9150,-3.2838) (-0.9093,-3.2368) (-0.9035,-3.1904) (-0.8978,-3.1446) (-0.8920,-3.0995) (-0.8863,-3.0549) (-0.8805,-3.0109) (-0.8748,-2.9675) (-0.8690,-2.9247) (-0.8633,-2.8824) (-0.8575,-2.8407) (-0.8518,-2.7996) (-0.8460,-2.7590) (-0.8403,-2.7189) (-0.8345,-2.6793) (-0.8288,-2.6403) (-0.8230,-2.6018) (-0.8173,-2.5637) (-0.8115,-2.5262) (-0.8058,-2.4892) (-0.8000,-2.4527) (-0.7942,-2.4166) (-0.7885,-2.3810) (-0.7827,-2.3459) (-0.7770,-2.3113) (-0.7712,-2.2771) (-0.7655,-2.2433) (-0.7597,-2.2100) (-0.7540,-2.1772) (-0.7482,-2.1447) (-0.7425,-2.1127) (-0.7367,-2.0812) (-0.7310,-2.0500) (-0.7252,-2.0192) (-0.7195,-1.9889) (-0.7137,-1.9589) (-0.7080,-1.9294) (-0.7022,-1.9002) (-0.6965,-1.8714) (-0.6907,-1.8430) (-0.6850,-1.8150) (-0.6792,-1.7873) (-0.6735,-1.7600) (-0.6677,-1.7331) (-0.6619,-1.7065) (-0.6562,-1.6802) (-0.6504,-1.6543) (-0.6447,-1.6288) (-0.6389,-1.6036) (-0.6332,-1.5787) (-0.6274,-1.5541) (-0.6217,-1.5299) (-0.6159,-1.5060) (-0.6102,-1.4824) (-0.6044,-1.4591) (-0.5987,-1.4361) (-0.5929,-1.4134) (-0.5872,-1.3911) (-0.5814,-1.3690) (-0.5757,-1.3472) (-0.5699,-1.3257) (-0.5642,-1.3044) (-0.5584,-1.2835) (-0.5527,-1.2628) (-0.5469,-1.2424) (-0.5412,-1.2223) (-0.5354,-1.2024) (-0.5296,-1.1828) (-0.5239,-1.1635) (-0.5181,-1.1444) (-0.5124,-1.1255) (-0.5066,-1.1069) (-0.5009,-1.0886) (-0.4951,-1.0705) (-0.4894,-1.0526) (-0.4836,-1.0350) (-0.4779,-1.0175) (-0.4721,-1.0004) (-0.4664,-0.9834) (-0.4606,-0.9667) (-0.4549,-0.9502) (-0.4491,-0.9339) (-0.4434,-0.9178) (-0.4376,-0.9019) (-0.4319,-0.8863) (-0.4261,-0.8708) (-0.4204,-0.8555) (-0.4146,-0.8405) (-0.4088,-0.8256) (-0.4031,-0.8109) (-0.3973,-0.7965) (-0.3916,-0.7822) (-0.3858,-0.7681) (-0.3801,-0.7541) (-0.3743,-0.7404) (-0.3686,-0.7268) (-0.3628,-0.7134) (-0.3571,-0.7002) (-0.3513,-0.6872) (-0.3456,-0.6743) (-0.3398,-0.6616) (-0.3341,-0.6490) (-0.3283,-0.6366) (-0.3226,-0.6244) (-0.3168,-0.6123) (-0.3111,-0.6004) (-0.3053,-0.5886) (-0.2996,-0.5770) (-0.2938,-0.5655) (-0.2881,-0.5542) (-0.2823,-0.5430) (-0.2765,-0.5319) (-0.2708,-0.5210) (-0.2650,-0.5102) (-0.2593,-0.4996) (-0.2535,-0.4891) (-0.2478,-0.4787) (-0.2420,-0.4684) (-0.2363,-0.4583) (-0.2305,-0.4483) (-0.2248,-0.4384) (-0.2190,-0.4286) (-0.2133,-0.4189) (-0.2075,-0.4094) (-0.2018,-0.3999) (-0.1960,-0.3906) (-0.1903,-0.3814) (-0.1845,-0.3723) (-0.1788,-0.3632) (-0.1730,-0.3543) (-0.1673,-0.3455) (-0.1615,-0.3368) (-0.1558,-0.3282) (-0.1500,-0.3197) (-0.1442,-0.3112) (-0.1385,-0.3029) (-0.1327,-0.2947) (-0.1270,-0.2865) (-0.1212,-0.2784) (-0.1155,-0.2704) (-0.1097,-0.2625) (-0.1040,-0.2547) (-0.0982,-0.2470) (-0.0925,-0.2393) (-0.0867,-0.2317) (-0.0810,-0.2242) (-0.0752,-0.2168) (-0.0695,-0.2094) (-0.0637,-0.2021) (-0.0580,-0.1948) (-0.0522,-0.1877) (-0.0465,-0.1806) (-0.0407,-0.1735) (-0.0350,-0.1666) (-0.0292,-0.1597) (-0.0235,-0.1528) (-0.0177,-0.1460) (-0.0119,-0.1393) (-0.0062,-0.1326) (-0.0004,-0.1260) (0.0053,-0.1194) (0.0111,-0.1129) (0.0168,-0.1065) (0.0226,-0.1001) (0.0283,-0.0937) (0.0341,-0.0874) (0.0398,-0.0812) (0.0456,-0.0750) (0.0513,-0.0688) (0.0571,-0.0627) (0.0628,-0.0566) (0.0686,-0.0506) (0.0743,-0.0446) (0.0801,-0.0387) (0.0858,-0.0328) (0.0916,-0.0269) (0.0973,-0.0211) (0.1031,-0.0153) (0.1088,-0.0095) (0.1146,-0.0038) (0.1204,0.0018) (0.1261,0.0075) (0.1319,0.0131) (0.1376,0.0186) (0.1434,0.0242) (0.1491,0.0297) (0.1549,0.0351) (0.1606,0.0406) (0.1664,0.0460) (0.1721,0.0513) (0.1779,0.0567) (0.1836,0.0620) (0.1894,0.0672) (0.1951,0.0725) (0.2009,0.0777) (0.2066,0.0829) (0.2124,0.0881) (0.2181,0.0932) (0.2239,0.0983) (0.2296,0.1034) (0.2354,0.1085) (0.2412,0.1135) (0.2469,0.1185) (0.2527,0.1235) (0.2584,0.1285) (0.2642,0.1334) (0.2699,0.1383) (0.2757,0.1432) (0.2814,0.1481) (0.2872,0.1529) (0.2929,0.1578) (0.2987,0.1626) (0.3044,0.1673) (0.3102,0.1721) (0.3159,0.1769) (0.3217,0.1816) (0.3274,0.1863) (0.3332,0.1910) (0.3389,0.1956) (0.3447,0.2003) (0.3504,0.2049) (0.3562,0.2095) (0.3619,0.2141) (0.3677,0.2187) (0.3735,0.2232) (0.3792,0.2277) (0.3850,0.2323) (0.3907,0.2368) (0.3965,0.2412) (0.4022,0.2457) (0.4080,0.2502) (0.4137,0.2546) (0.4195,0.2590) (0.4252,0.2634) (0.4310,0.2678) (0.4367,0.2722) (0.4425,0.2765) (0.4482,0.2809) (0.4540,0.2852) (0.4597,0.2895) (0.4655,0.2938) (0.4712,0.2981) (0.4770,0.3024) (0.4827,0.3066) (0.4885,0.3109) (0.4942,0.3151) (0.5000,0.3193)};
\draw[confcol,line width=1.25pt] plot coordinates {(-1.4500,0.5351) (-1.4442,0.5308) (-1.4385,0.5264) (-1.4327,0.5221) (-1.4270,0.5177) (-1.4212,0.5133) (-1.4155,0.5089) (-1.4097,0.5045) (-1.4040,0.5001) (-1.3982,0.4956) (-1.3925,0.4912) (-1.3867,0.4867) (-1.3810,0.4822) (-1.3752,0.4777) (-1.3695,0.4731) (-1.3637,0.4686) (-1.3580,0.4640) (-1.3522,0.4594) (-1.3465,0.4548) (-1.3407,0.4502) (-1.3350,0.4456) (-1.3292,0.4409) (-1.3235,0.4362) (-1.3177,0.4316) (-1.3119,0.4268) (-1.3062,0.4221) (-1.3004,0.4173) (-1.2947,0.4126) (-1.2889,0.4078) (-1.2832,0.4030) (-1.2774,0.3981) (-1.2717,0.3932) (-1.2659,0.3884) (-1.2602,0.3835) (-1.2544,0.3785) (-1.2487,0.3736) (-1.2429,0.3686) (-1.2372,0.3636) (-1.2314,0.3586) (-1.2257,0.3535) (-1.2199,0.3484) (-1.2142,0.3433) (-1.2084,0.3382) (-1.2027,0.3330) (-1.1969,0.3278) (-1.1912,0.3226) (-1.1854,0.3174) (-1.1796,0.3121) (-1.1739,0.3068) (-1.1681,0.3015) (-1.1624,0.2961) (-1.1566,0.2907) (-1.1509,0.2853) (-1.1451,0.2799) (-1.1394,0.2744) (-1.1336,0.2689) (-1.1279,0.2633) (-1.1221,0.2577) (-1.1164,0.2521) (-1.1106,0.2464) (-1.1049,0.2407) (-1.0991,0.2350) (-1.0934,0.2292) (-1.0876,0.2234) (-1.0819,0.2176) (-1.0761,0.2117) (-1.0704,0.2058) (-1.0646,0.1998) (-1.0588,0.1938) (-1.0531,0.1877) (-1.0473,0.1816) (-1.0416,0.1755) (-1.0358,0.1693) (-1.0301,0.1630) (-1.0243,0.1567) (-1.0186,0.1504) (-1.0128,0.1440) (-1.0071,0.1376) (-1.0013,0.1311) (-0.9956,0.1245) (-0.9898,0.1179) (-0.9841,0.1113) (-0.9783,0.1046) (-0.9726,0.0978) (-0.9668,0.0910) (-0.9611,0.0841) (-0.9553,0.0771) (-0.9496,0.0701) (-0.9438,0.0630) (-0.9381,0.0559) (-0.9323,0.0487) (-0.9265,0.0414) (-0.9208,0.0340) (-0.9150,0.0266) (-0.9093,0.0191) (-0.9035,0.0115) (-0.8978,0.0039) (-0.8920,-0.0038) (-0.8863,-0.0116) (-0.8805,-0.0195) (-0.8748,-0.0275) (-0.8690,-0.0355) (-0.8633,-0.0436) (-0.8575,-0.0519) (-0.8518,-0.0602) (-0.8460,-0.0686) (-0.8403,-0.0771) (-0.8345,-0.0857) (-0.8288,-0.0943) (-0.8230,-0.1031) (-0.8173,-0.1120) (-0.8115,-0.1210) (-0.8058,-0.1301) (-0.8000,-0.1392) (-0.7942,-0.1485) (-0.7885,-0.1580) (-0.7827,-0.1675) (-0.7770,-0.1771) (-0.7712,-0.1868) (-0.7655,-0.1967) (-0.7597,-0.2067) (-0.7540,-0.2168) (-0.7482,-0.2270) (-0.7425,-0.2374) (-0.7367,-0.2479) (-0.7310,-0.2585) (-0.7252,-0.2692) (-0.7195,-0.2801) (-0.7137,-0.2911) (-0.7080,-0.3023) (-0.7022,-0.3135) (-0.6965,-0.3250) (-0.6907,-0.3366) (-0.6850,-0.3483) (-0.6792,-0.3602) (-0.6735,-0.3722) (-0.6677,-0.3844) (-0.6619,-0.3967) (-0.6562,-0.4092) (-0.6504,-0.4219) (-0.6447,-0.4347) (-0.6389,-0.4477) (-0.6332,-0.4609) (-0.6274,-0.4742) (-0.6217,-0.4878) (-0.6159,-0.5014) (-0.6102,-0.5153) (-0.6044,-0.5294) (-0.5987,-0.5436) (-0.5929,-0.5580) (-0.5872,-0.5726) (-0.5814,-0.5874) (-0.5757,-0.6024) (-0.5699,-0.6176) (-0.5642,-0.6330) (-0.5584,-0.6486) (-0.5527,-0.6645) (-0.5469,-0.6805) (-0.5412,-0.6967) (-0.5354,-0.7132) (-0.5296,-0.7298) (-0.5239,-0.7467) (-0.5181,-0.7638) (-0.5124,-0.7812) (-0.5066,-0.7987) (-0.5009,-0.8165) (-0.4951,-0.8346) (-0.4894,-0.8529) (-0.4836,-0.8714) (-0.4779,-0.8901) (-0.4721,-0.9092) (-0.4664,-0.9284) (-0.4606,-0.9480) (-0.4549,-0.9678) (-0.4491,-0.9878) (-0.4434,-1.0081) (-0.4376,-1.0287) (-0.4319,-1.0496) (-0.4261,-1.0707) (-0.4204,-1.0922) (-0.4146,-1.1139) (-0.4088,-1.1359) (-0.4031,-1.1582) (-0.3973,-1.1807) (-0.3916,-1.2036) (-0.3858,-1.2268) (-0.3801,-1.2503) (-0.3743,-1.2742) (-0.3686,-1.2983) (-0.3628,-1.3228) (-0.3571,-1.3475) (-0.3513,-1.3727) (-0.3456,-1.3981) (-0.3398,-1.4239) (-0.3341,-1.4500) (-0.3283,-1.4765) (-0.3226,-1.5033) (-0.3168,-1.5305) (-0.3111,-1.5581) (-0.3053,-1.5860) (-0.2996,-1.6143) (-0.2938,-1.6430) (-0.2881,-1.6720) (-0.2823,-1.7015) (-0.2765,-1.7313) (-0.2708,-1.7615) (-0.2650,-1.7922) (-0.2593,-1.8232) (-0.2535,-1.8547) (-0.2478,-1.8866) (-0.2420,-1.9189) (-0.2363,-1.9516) (-0.2305,-1.9848) (-0.2248,-2.0184) (-0.2190,-2.0524) (-0.2133,-2.0869) (-0.2075,-2.1219) (-0.2018,-2.1574) (-0.1960,-2.1933) (-0.1903,-2.2297) (-0.1845,-2.2665) (-0.1788,-2.3039) (-0.1730,-2.3418) (-0.1673,-2.3801) (-0.1615,-2.4190) (-0.1558,-2.4584) (-0.1500,-2.4984) (-0.1442,-2.5388) (-0.1385,-2.5798) (-0.1327,-2.6214) (-0.1270,-2.6635) (-0.1212,-2.7061) (-0.1155,-2.7493) (-0.1097,-2.7931) (-0.1040,-2.8375) (-0.0982,-2.8825) (-0.0925,-2.9281) (-0.0867,-2.9743) (-0.0810,-3.0211) (-0.0752,-3.0685) (-0.0695,-3.1166) (-0.0637,-3.1653) (-0.0580,-3.2146) (-0.0522,-3.2646) (-0.0465,-3.3153) (-0.0407,-3.3667) (-0.0350,-3.4187) (-0.0292,-3.4714) (-0.0235,-3.5249) (-0.0177,-3.5790) (-0.0119,-3.6339) (-0.0062,-3.6895) (-0.0004,-3.7458) (0.0053,-3.8029) (0.0111,-3.8608) (0.0168,-3.9194) (0.0226,-3.9788) (0.0283,-4.0390) (0.0341,-4.1000) (0.0398,-4.1618) (0.0456,-4.2244) (0.0513,-4.2879) (0.0571,-4.3522) (0.0628,-4.4174) (0.0686,-4.4835) (0.0743,-4.5504) (0.0801,-4.6182) (0.0858,-4.6869) (0.0916,-4.7566) (0.0973,-4.8271) (0.1031,-4.8986) (0.1088,-4.9711) (0.1146,-5.0445) (0.1204,-5.1189) (0.1261,-5.1943) (0.1319,-5.2708) (0.1376,-5.3482) (0.1434,-5.4266) (0.1491,-5.5061) (0.1549,-5.5867) (0.1606,-5.6683) (0.1664,-5.7511) (0.1721,-5.8349) (0.1779,-5.9198) (0.1836,-6.0059) (0.1894,-6.0931) (0.1951,-6.1815)};
\fill[gray!50] (-0.4577,-0.9581) circle (0.011 and 0.092);

\node[gray!88,font=\footnotesize,anchor=north,align=center] at (-0.45,0.86)
      {transition window\\[-1.2pt]
       };

\node[diffcol,font=\footnotesize,anchor=north,align=center] at (-1.06,-6.78)
      {$|\Phi_{c}(0)-f_{M}(0)|\asymp e^{-1/c}$\\[-1pt]
       \\[-1pt]
       \textcolor{gray}{diffusive window: $\Phi_{c}\simeq f_{M}$}};
\node[confcol,font=\footnotesize,anchor=north,align=center] at (0.16,-6.78)
      {$\|c^{-1/2}\Phi_{c}(y/\sqrt c)-\mathcal U_{M}\|_{L^{\infty}(0,1)}\asymp e^{-\pi^{2}c}$\\[-1pt]
       \\[-1pt]
       \textcolor{gray}{confinement window: $\Phi_{c}\simeq\mathcal U_{M}$}};

\end{tikzpicture}
\caption{Illustration of Theorem~\ref{thm:intro-main} at $M=1$: the deviation
of $\Phi_{c}$ from each of the two states it degenerates to, evaluated from the
closed form \eqref{eq:Phic} and so free of discretisation error.  On the left,
the deviation from the self-similar profile at the origin, in similarity
variables, obeying the exact asymptotics \eqref{eq:sharpasymp}; on the right,
the deviation from the interval equilibrium in domain variables, at the
spectral rate $e^{-\pi^{2}c}$ of Proposition~\ref{prop:cinfty}.  The two are
measured in different coordinates because they are taken in different
coordinates (Remark~\ref{rem:scaling}), and the uniform-in-$L$ comparison
carries the exponent of the left curve.  Each deviation is small exactly on the
window where the corresponding description is valid; the two cross near
$c\approx0.35$, and in between only $\Phi_{c}$ describes the solution.  The
shading marks the transition window, in which there is no threshold, the
passage between regimes being continuous (Remark~\ref{rem:regimes}).}
\label{fig:crossover}
\end{figure}
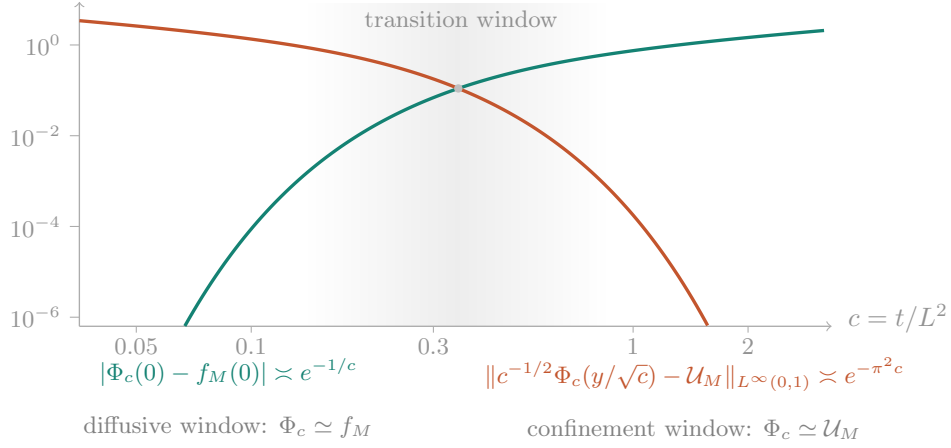

Theorem~\ref{thm:intro-main} is a uniform singular-limit statement: in
domain variables the datum occupies a layer of width $O(L^{-1})$, the
observation time is $cL^2$, and the nonlinear observable is $w_x/w$.
Section~\ref{FI} is preparatory and essentially known: Theorems~\ref{T1} and \ref{expt1} are the $L^{1}$ version of results proved in \cite{Watanabe2016} for $C^{2}$ data, and the half-line analysis is close to \cite{Karch2000}. It is included because the whole paper is read off the Hopf--Cole reduction recalled there.

Section~\ref{sm} proves Theorem~\ref{thm:intro-main}.
Section~\ref{sec:variable} asks how much survives when the diffusivity depends
on $x$, $u_t=\partial_x\big(a(x)(u_x+u^2)\big)$. Hopf--Cole still linearizes and
the picture separates: the stationary profile $U_M^L$ is unchanged and the
relaxation rate is governed by the principal eigenvalue $\lambda_1(a)$, but the
explicit crossover is lost, the family $\Phi_c$ and the sharp constant resting
on the closed-form image representation available only for constant
coefficients; the half-line regime there is close in spirit to Duro and Zuazua
\cite{DuroZuazua1999}. Section~\ref{NS} gives numerical illustrations, and
Section~\ref{multi} closes with a brief outlook.

\section{Existence, uniqueness, and large-time behavior}\label{FI}

This section collects, for completeness, the well-posedness and large-time facts used in
Section~\ref{sm}. Nothing here is essentially new: Theorems~\ref{T1} and \ref{expt1} are the
$L^1$ version of results proved in \cite{Watanabe2016} for $C^2$ data, and
Section~\ref{infinite} is close in spirit to \cite{Karch2000}, what is specific to the
conservative condition being the explicit form of the profile $f_M$. Proofs are included
because the whole paper is read off the Hopf--Cole reduction recalled here.

Throughout, a \emph{Hopf--Cole solution} of \eqref{eb} or \eqref{ei} means a
function $u=w_x/w$, where $w$ is the classical solution of the associated linear
problem and $u(\cdot,t)\to u_0$ in $L^1$ as $t\downarrow0$.  All uniqueness
statements below refer to this class.  We use this terminology explicitly to
avoid asserting uniqueness in a larger, unspecified class of weak solutions;
the transition analysis requires precisely the Hopf--Cole solution and its heat
kernel estimates.

\subsection{\texorpdfstring{Large-time behavior in $(0,L)$}{Large-time behavior in (0,L)}}

We consider the Burgers equation posed on a bounded interval:
\begin{equation}\label{eb}\left\{\begin{array}{l}
u_t-u_{xx}=(u^2)_x,\ \ \ \  (x,t)\in(0,L)\times (0,\infty)\\
u_x(0,t)=-u^2(0,t),\ \ u_x(L,t)=-u^2(L,t),\ \ \  \ t\in (0,\infty)\\
u(x,0)=u_0(x),\ \ \ x\in (0,L).
\end{array}\right.
\end{equation}
As already mentioned, \eqref{eb} reads $u_t-\partial_x(u_x+u^2)=0$, so the flux is
$-(u_x+u^2)$ and the boundary conditions are precisely the statement that it vanishes at
$x=0,L$.

\begin{teo}\label{T1}
Let $u_0 \in L^1(0,L)$ have mass $M$. Then problem \eqref{eb} admits a unique
Hopf--Cole solution, and
it preserves the mass,
\[
\int_0^L u(x,t)\,dx = M, \qquad t>0.
\]
\end{teo}

\begin{proof}
Set $v(x,t)=\int_0^x u(z,t)\,dz$, so that $u=v_x$ and $v(0,t)=0$. Since \eqref{eb} in conservative form reads $u_t=\partial_x(u_x+u^2)$, for every $t>0$
\[
\partial_t v(L,t)=\int_0^L u_t\,dx=\int_0^L \partial_x(u_x+u^2)\,dx=\big[u_x+u^2\big]_{x=0}^{x=L}=0,
\]
the last equality being exactly the boundary conditions of \eqref{eb}; hence $v(L,t)=v(L,0)=M$ for all $t>0$. Differentiating $v=\int_0^x u$ and using \eqref{eb} shows that $v$ then satisfies
\begin{equation}\label{ebv}
\left\{
\begin{array}{ll}
v_t-v_{xx}=(v_x)^2, & (x,t)\in(0,L)\times(0,\infty),\\[1mm]
v(0,t)=0,\qquad v(L,t)=M, & t>0,\\[1mm]
v(x,0)=\displaystyle\int_0^x u_0(z)\,dz, & x\in(0,L).
\end{array}
\right.
\end{equation}
Conversely, if $v$ solves \eqref{ebv}, then $v(L,\cdot)\equiv M$ gives $v_t(L,t)=0$, and evaluating $v_t-v_{xx}=(v_x)^2$ at $x=L$ yields $v_{xx}(L,t)=-(v_x(L,t))^2$, i.e.\ $u_x(L,t)=-u^2(L,t)$; likewise at $x=0$. Thus the nonlinear Neumann conditions of \eqref{eb} and the Dirichlet problem \eqref{ebv} are equivalent, and mass conservation is a consequence rather than an assumption.

Applying the Hopf--Cole transformation $w=e^v$,
\begin{equation}\label{ebw}
\left\{
\begin{array}{ll}
w_t-w_{xx}=0,& (x,t)\in (0,L)\times(0,\infty)\\[1mm]
w(0,t)=1,\ \  w(L,t)=e^M, & t>0,\\[1mm]
w(x,0)=w_0(x):=\displaystyle \exp\!\left(\int_0^x u_0(z)\,dz\right), & x\in(0,L).
\end{array}
\right.
\end{equation}
Since $u_0\in L^1(0,L)$, the datum $v(\cdot,0)$ is absolutely continuous, so $w_0$ is
continuous on $[0,L]$, strictly positive, and compatible with the boundary data:
$w_0(0)=1$ and $w_0(L)=e^M$. Subtracting the affine steady state $W(x)=1+(e^M-1)x/L$ reduces
\eqref{ebw} to the heat equation with homogeneous Dirichlet conditions and continuous datum
$w_0-W$ vanishing at both endpoints, for which existence and uniqueness of a classical
solution, continuous up to $t=0$, are standard. The maximum principle gives $w>0$; more precisely $0<m_{0}\le w\le M_{0}<\infty$ on
$[0,L]\times[0,\infty)$ with the explicit constants of \eqref{eq:wbounds} below, and the
same bounds hold for $w_{0}$. In particular $u:=w_x/w$ is well defined and smooth for
$t>0$.

It remains to identify the initial trace. Since $u_{0}\in L^{1}(0,L)$, the datum
$w_{0}$ is absolutely continuous with $w_{0}'=w_{0}u_{0}\in L^{1}(0,L)$. Let
$h(\cdot,t)$ solve the heat equation on $(0,L)$ with homogeneous Neumann
conditions and datum $w_{0}'$. Then $\tilde w(x,t):=1+\int_{0}^{x}h(z,t)\,dz$
satisfies $\tilde w_{t}=h_{x}=\tilde w_{xx}$, has $\tilde w(0,t)=1$ and, the
Neumann semigroup preserving the integral, $\tilde w(L,t)=w_{0}(L)=e^{M}$, with
$\tilde w(\cdot,0)=w_{0}$; so $\tilde w$ solves \eqref{ebw} and, by the
uniqueness just recalled, $w=\tilde w$ and $w_{x}=h$. That semigroup being
strongly continuous on $L^{1}(0,L)$, one has $w_{x}(\cdot,t)\to w_{0}u_{0}$ in
$L^{1}(0,L)$ and $w(\cdot,t)\to w_{0}$ uniformly on $[0,L]$. Writing
$u(\cdot,t)-u_{0}=(w_{x}-w_{0}u_{0})/w+w_{0}u_{0}\big(w^{-1}-w_{0}^{-1}\big)$
and using $w\ge m_{0}>0$, both terms tend to $0$ in $L^{1}$ as $t\to0^{+}$, so
$u(\cdot,t)\to u_{0}$ in $L^{1}(0,L)$.

This is the unique solution in the class fixed above, since the correspondence
$u\leftrightarrow v\leftrightarrow w$ is one-to-one and $w$ is unique.

Mass conservation is immediate from \eqref{ebv}: for every $t>0$,
\[
\int_0^L u(x,t)\,dx=v(L,t)-v(0,t)=M .
\]
\end{proof}

When the interval length varies, we write the stationary solution of \eqref{eb} with mass $M$ as
\begin{equation}\label{UM}
U_M^L(x)=\dfrac{e^M-1}{x(e^M-1)+L}.
\end{equation}
Its unit-interval rescaling is
\begin{equation}\label{eq:calUM}
\mathcal U_M(y):=L U_M^L(Ly)=\frac{e^M-1}{1+(e^M-1)y},\qquad 0\le y\le1.
\end{equation}

The stationary solution is unique, and the uniqueness is transparent at the linear
level: a stationary $w$ solves $w_{xx}=0$ with $w(0)=1$, $w(L)=e^M$, whose only
solution is the affine profile $W(x)=1+(e^M-1)x/L>0$. Through the bijection
$u\leftrightarrow v\leftrightarrow w$ this determines the stationary states at the
other two levels uniquely, $V=\log W$ and
\[
U_M^L(x)=\frac{W_x}{W}=\frac{e^M-1}{L+(e^M-1)x},
\]
which is \eqref{UM}. In particular $U_M^L$ is the only possible equilibrium of
\eqref{eb} in the mass class $M$: since the mass is conserved along the evolution,
every trajectory issuing from a datum of mass $M$ can relax to no steady state
other than $U_M^L$, which is what makes the long-time limit unambiguous.

\begin{teo}\label{expt1}
Let $u_0 \in L^1(0,L)$ have mass $M$, let $u$ be the corresponding solution of \eqref{eb},
and let $\lambda_1=\pi^2/L^2$ be the first Dirichlet eigenvalue of $(0,L)$. Then, for every
$1\le p\le\infty$ there is $C=C(\|u_0\|_{L^1},L,p)>0$ such that
\begin{equation}\label{eq:expconv}
\|u(\cdot,t)-U_M^L\|_{L^p(0,L)}
\le C\,\bigl(1+t^{-3/4}\bigr)\, e^{-\lambda_1 t},
\qquad t>0 .
\end{equation}
\end{teo}

\begin{remark}\label{rem:expt1}
The prefactor cannot be dispensed with: for $u_0(x)=x^{-1/2}$ on $(0,1)$, which is
admissible ($\|u_0\|_{L^1}=2$, $w_0=e^{2\sqrt x}$), one has $w_0\notin H^1$, hence
$\|u(\cdot,t)\|_{L^2}\to\infty$ as $t\to0^+$ while $U_M^L\in L^\infty$, so no bound
$Ce^{-\lambda_1 t}$ can hold up to $t=0$ for $p>1$. Note also that
$\|U_M^L\|_{L^\infty}\sim1/L$ and $\|u(\cdot,t)\|_{L^\infty}\lesssim t^{-1/2}$, so at
$t\sim L^2$ both sides of \eqref{eq:expconv} are $O(1/L)$; the estimate is informative in a
relative sense only for $t\gg L^2$, which is the regime in which it is used
(Remark~\ref{rem:regimes}(3)).
\end{remark}

\begin{proof}
With $v$ and $w$ as in Theorem~\ref{T1}, $w$ solves \eqref{ebw} and $u=w_x/w$; the
stationary profile is $U_M^L=W_x/W$ with $W(x)=1+(e^M-1)x/L$. Set $z:=w-W$, which solves the
heat equation on $(0,L)$ with homogeneous Dirichlet conditions and datum $z(\cdot,0)=w_0-W$.
By \eqref{ebw} and the maximum principle,
\begin{equation}\label{eq:wbounds}
0<m_0:=\min\{1,e^M\}\wedge e^{-\|u_0\|_{L^1}}\le w\le e^{\|u_0\|_{L^1}}\vee\max\{1,e^M\},
\end{equation}
and the same bounds hold for $W$. Since $z(\cdot,0)\in L^\infty(0,L)\subset L^2(0,L)$, the
smoothing of the Dirichlet heat semigroup together with its spectral gap $\lambda_1$ gives,
for $j=0,1$,
\begin{equation}\label{eq:zsmoothing}
\|\partial_x^j z(\cdot,t)\|_{L^\infty(0,L)}
\le C\,\bigl(1+t^{-(2j+1)/4}\bigr)\,e^{-\lambda_1 t}\,\|z(\cdot,0)\|_{L^2(0,L)},
\qquad t>0,
\end{equation}
with $C=C(L)$. The two regimes are treated separately: for $0<t\le1$ the factor
$t^{-(2j+1)/4}$ is the standard $L^2\to W^{j,\infty}$ smoothing of the Dirichlet
semigroup, while for $t\ge1$ one first evolves for one unit of time and then applies
the spectral gap on $[1,t]$, which gives the bound with a constant prefactor. A factor
$t^{-(2j+1)/4}$ alone would fail for large $t$, as the eigenfunction
$z(\cdot,0)=\sin(\pi\cdot/L)$ shows. Finally,
\[
u-U_M^L=\frac{w_x}{w}-\frac{W_x}{W}=\frac{z_x}{w}-\frac{W_x}{wW}\,z ,
\]
so by \eqref{eq:wbounds} and $\|W\|_{C^1}\le C(L)$,
\[
|u-U_M^L|\le C\bigl(|z_x|+|z|\bigr) .
\]
Combining with \eqref{eq:zsmoothing} and $\|\cdot\|_{L^p(0,L)}\le L^{1/p}\|\cdot\|_{L^\infty(0,L)}$
gives \eqref{eq:expconv} for every $p\in[1,\infty]$; here the exponent
$3/4$ is the value furnished by the $j=1$ smoothing estimate in
\eqref{eq:zsmoothing}.
\end{proof}

\subsection{\texorpdfstring{Large-time behavior in $(0,\infty)$}{Large-time behavior in (0,infty)}}

We now consider the same problem on the half-line:
\begin{equation}\label{ei}\left\{\begin{array}{l}
u_t-u_{xx}=(u^2)_x,\ \ \ \  (x,t)\in(0,\infty)\times (0,\infty)\\
u_x(0,t)=-u^2(0,t),\ \ \  \ t\in (0,\infty)\\
u(x,0)=u_0(x)\in L^{1}(0,\infty).
\end{array}\right.
\end{equation}
The results of this subsection extend to the half-line the classical theory of large-time
behavior for convection-diffusion equations on the whole space \cite{evz1,evz2,ez}; as previously mentioned, the
half-line case with a Neumann condition was studied by Biler and Karch \cite{Karch2000}. The
systematic study of convergence to steady states for parabolic equations was initiated by
Friedman \cite{Friedman1959,Friedman1961}.

\begin{teo}\label{TI}
Let $u_0 \in L^1(0,\infty)$ have mass $M$. Then problem \eqref{ei} admits a unique
Hopf--Cole solution,
it preserves the mass, and for every $p\ge1$,
\[
\|u(\cdot,t)\|_{L^p} \le C\, t^{-\frac12(1-\frac1p)}, \qquad t>0.
\]
\end{teo}

\begin{proof}
As in Theorem~\ref{T1}, set $w=\exp\bigl(\int_0^x u\bigr)$, so that $u=w_x/w$ and $w$ satisfies
\[
\left\{
\begin{array}{ll}
w_t-w_{xx}=0, &(x,t)\in(0,\infty)\times(0,\infty),\\[1mm]
w(0,t)=1, & t>0,\\[1mm]
w(x,0)=w_0(x):=\displaystyle\exp\!\left(\int_0^x u_0(z)\,dz\right), & x>0.
\end{array}
\right.
\]
The datum $w_0$ is continuous, strictly positive, bounded, and $w_0(0)=1$; existence,
uniqueness and positivity follow as before \cite{ev,Lieberman1996}, and the initial trace
$u(\cdot,t)\to u_{0}$ in $L^{1}(0,\infty)$ is obtained exactly as in Theorem \ref{T1},
with the Neumann semigroup of $(0,L)$ replaced by its half-line counterpart.

For the mass, set $h=w-1$, which has homogeneous Dirichlet boundary data. Then
$h(x,t)=\int_0^\infty G_D(t,x,y)(w_0(y)-1)\,dy$. Equivalently,
\[
w(x,t)-e^M=\int_0^\infty G_D(t,x,y)(w_0(y)-e^M)\,dy
+(e^M-1)\left(\int_0^\infty G_D(t,x,y)\,dy-1\right).
\]
Both terms tend to zero as $x\to\infty$. Indeed, for
any $\varepsilon>0$ choose $A$ with $|w_0-e^M|\le\varepsilon$ on $[A,\infty)$; for
$x\ge2A$ the contribution of $[0,A]$ to $\int G_D(t,x,y)(w_0(y)-e^M)\,dy$ is
$O(e^{-x^2/(16t)})$ and the rest is $\le\varepsilon$, whence
$\lim_{x\to\infty}w(x,t)=e^M$ for every $t>0$; the boundary term tends to zero by the same kernel formula. Since $w(0,t)=1$,
\[
\int_0^\infty u(x,t)\,dx=\ln w(\infty,t)-\ln w(0,t)=M,
\qquad t>0.
\]

For the decay, set $h:=w_x$. Differentiating the Dirichlet condition $w(0,t)=1$ in $t$ and
using the equation gives $w_{xx}(0,t)=0$, that is $h_x(0,t)=0$; equivalently, since
$h_x=w(u_x+u^2)$, this is the conservative condition $u_x(0,t)=-u^2(0,t)$. Thus $h$ solves
the heat equation on $(0,\infty)$ with a homogeneous Neumann condition at $x=0$. Extending
$h$ evenly across $x=0$ yields a solution on $\mathbb R$, to which the standard heat
semigroup estimates apply:
\[
\|h(\cdot,t)\|_{L^p(0,\infty)}
\le
Ct^{-\frac12(1-\frac1p)}\|h(\cdot,0)\|_{L^1(0,\infty)},
\qquad 1\le p\le\infty.
\]
Since $h(x,0)=w_0(x)u_0(x)$ and $e^{-\|u_0\|_{L^1}}\le w_0\le e^{\|u_0\|_{L^1}}$, we get
$\|h(\cdot,0)\|_{L^1}\le C\|u_0\|_{L^1}$; note also
$\int_0^\infty h(\cdot,0)=w_0(\infty)-w_0(0)=e^M-1$. The maximum principle gives
$0<c_0\le w\le C_0$ at positive times, so $\|u(\cdot,t)\|_{L^p}\le C\|h(\cdot,t)\|_{L^p}$,
which is the claim.
\end{proof}

\subsection{Self-similar asymptotic behavior}\label{infinite}

Equation \eqref{ei} is invariant under $u_\lambda(x,t)=\lambda u(\lambda x,\lambda^2t)$, and
solutions invariant under this scaling have the form $u(x,t)=t^{-1/2}f(x/\sqrt t)$. We
characterize the profiles $f$ explicitly.

\begin{prop}\label{TI2}
For every $M \in \mathbb{R}$ there is $f_M \in H^1(0,\infty)$ with
$\int_0^{\infty} f_M = M$ such that
\[
u_M(x,t) = t^{-1/2}\, f_M\!\left(\frac{x}{\sqrt{t}}\right)
\]
solves
\begin{equation}\label{ed}
u_t-u_{xx}=(u^2)_x \quad\text{in } (0,\infty)\times (0,\infty),
\qquad
u_x(0,t)=-u^2(0,t),\quad t>0 .
\end{equation}
Explicitly,
\[
f_M(\xi)=\frac{(e^M-1)\,b(\xi)}{1+(e^M-1)\displaystyle\int_0^\xi b},
\qquad b(\xi)=\pi^{-1/2}e^{-\xi^2/4},
\]
with $f_0\equiv0$. Moreover $u_M(\cdot,t)\rightharpoonup M\delta_0$ weakly-$*$ as $t\to0^+$.
\end{prop}

\begin{proof}
Substituting $u(x,t)=t^{-1/2}f(\xi)$, $\xi=x/\sqrt t$, the profile must satisfy
\begin{equation}\label{ef}
-f''-\frac{\xi}{2}f'-\frac12 f=(f^2)',
\qquad \xi>0,
\qquad
\int_0^\infty f=M .
\end{equation}
Set $g(\xi)=\exp\bigl(\int_0^\xi f\bigr)$, so $g>0$, $g(0)=1$, $g(\infty)=e^M$ and $f=g'/g$. A direct
computation gives
\[
-g''-\frac{\xi}{2}g'=-\Bigl(f'+f^2+\frac{\xi}{2}f\Bigr)g ,
\]
while integrating \eqref{ef} once yields $f'+f^2+\frac{\xi}{2}f=\kappa$ with
$\kappa=f'(0)+f(0)^2$, so that
\[
-g''-\frac{\xi}{2}g'=-\kappa g .
\]
Here $\kappa=0$: it is exactly the conservative condition $f'(0)=-f(0)^2$, and it is forced at infinity. Integrating $(e^{\xi^2/4}g')'=\kappa\,e^{\xi^2/4}g$ gives $e^{\xi^2/4}g'(\xi)=g'(0)+\kappa\int_0^\xi e^{s^2/4}g$. If $\kappa\neq0$ and $g$ were bounded with $g\to e^M$, the right-hand side would give $g'(\xi)\sim 2\kappa e^M/\xi$, hence $g(\xi)\sim 2\kappa e^M\log\xi\to\pm\infty$, contradicting $g(\infty)=e^M$. Thus $\kappa=0$.

Hence $-g''-\frac{\xi}{2}g'=0$ with $g(0)=1$, $g(\infty)=e^M$. Its solution is
\[
g(\xi)=1+(e^M-1)\int_0^\xi b,
\qquad b(\xi)=\pi^{-1/2}e^{-\xi^2/4},
\]
since $\int_0^\infty b=1$; here $b$ is the unit-mass self-similar profile of the heat
equation, characterized by $-b''-\frac{\xi}{2}b'-\frac12b=0$. Conversely, defining $g$ by
this formula and $f_M:=g'/g$, one checks that $g>0$: as $\int_0^\xi b$ increases from $0$ to
$1$, $g$ ranges monotonically between $1$ and $e^M$, so $g\ge\min(1,e^M)>0$ for every real
$M$. Reversing the computation above shows that $f_M$ solves \eqref{ef}, and
$\int_0^\infty f_M=\ln g(\infty)-\ln g(0)=M$. For $M=0$ the formula gives $g\equiv1$ and
$f_0\equiv0$, consistently.

Finally, for $\varphi\in C_c([0,\infty))$, the change of variable $x=\xi\sqrt t$ gives
$\int_0^\infty u_M(x,t)\varphi(x)\,dx=\int_0^\infty f_M(\xi)\varphi(\xi\sqrt t)\,d\xi\to
\varphi(0)\int_0^\infty f_M=M\varphi(0)$ as $t\to0^+$, by dominated convergence, since
$f_M\in L^1$.
\end{proof}

\begin{teo}\label{is}
Let $u_0\in L^1(0,\infty)$ with mass $M$, let $u$ solve \eqref{ei} and let $u_M$ be as in
Proposition~\ref{TI2}. Then
\begin{equation}\label{dss}
t^{\frac12\left(1-\frac1p\right)}
\|u(\cdot,t)-u_M(\cdot,t)\|_{L^p(0,\infty)}
\longrightarrow0,
\qquad t\to\infty,
\end{equation}
for all $1\le p\le\infty$.
\end{teo}

\begin{proof}
Let $w$ be the Hopf--Cole transform of $u$ and set
$w_M(x,t)=1+(e^M-1)\int_0^{x/\sqrt t}b$, so that
$(w_M)_x=(e^M-1)t^{-1/2}b(x/\sqrt t)$ and $u_M=(w_M)_x/w_M$. As in the proof of
Theorem~\ref{TI}, $h:=w_x$ solves the heat equation on $(0,\infty)$ with
$h_x(0,t)=0$ and $\int_0^\infty h(\cdot,0)=e^M-1$; its even extension across
$x=0$ solves the heat equation on $\mathbb R$ with mass $2(e^M-1)$, so the
classical whole-line asymptotics give convergence, after rescaling, to
$2(e^M-1)$ times the unit Gaussian, which restricted to $x>0$ is exactly
$(w_M)_x$. Hence
$t^{\frac12(1-\frac1p)}\|w_x(\cdot,t)-(w_M)_x(\cdot,t)\|_{L^p(0,\infty)}\to0$ for
$1\le p\le\infty$. By the maximum principle $0<c\le w,w_M\le C$ uniformly, and
\[
u-u_M=\frac{w_x-(w_M)_x}{w}-\frac{(w_M)_x}{w\,w_M}\,(w-w_M).
\]
The first term is handled by the estimate just established. For the second, $w-w_M=\int_0^x(w_y-(w_M)_y)$
gives $\|w(\cdot,t)-w_M(\cdot,t)\|_{L^\infty}\le\|w_x-(w_M)_x\|_{L^1}\to0$, while
$t^{\frac12(1-\frac1p)}(w_M)_x$ is bounded in $L^p$. This is \eqref{dss}.
\end{proof}

Theorem~\ref{is} gives no rate.  For compactly supported data the next two
results first establish the $t^{-1/2}$ rate in $C^0$ and then propagate the
same rate to every local $C^k$ norm, as required in Section~\ref{sm}.

\begin{teo}\label{thm:halflinerate}
Let $u_0\in L^1(0,\infty)$ with mass $M$ and $\operatorname{supp}u_0\subset[0,R_0]$, and
let $v^\infty(x,t)=t^{1/2}u(x\sqrt t,t)$ be the rescaled solution of \eqref{ei}. Then,
for every $R>0$ there is $C=C(\|u_0\|_{L^1},R_0,R)>0$ such that
\begin{equation}\label{eq:hlrate-statement}
\|v^\infty(\cdot,t)-f_M\|_{L^\infty(0,R)}\le C\,t^{-1/2},
\qquad t\ge1 .
\end{equation}
\end{teo}

\begin{proof}
Let $w$ and $w_M$ be the Hopf--Cole transforms of $u$ and of the self-similar solution
$u_M$, and set $d:=w-w_M$. Then $d$ solves the homogeneous Dirichlet heat equation with
initial datum $d_0(x)=w_0(x)-e^M$ on $x>0$, which is compactly supported. Its odd extension belongs to $L^1(\mathbb R)$. Standard heat-semigroup
derivative estimates applied to that extension give
\[
\|w_x(\cdot,t)-(w_M)_x(\cdot,t)\|_{L^1(0,\infty)}=O(t^{-1/2}),
\qquad
\|w_x(\cdot,t)-(w_M)_x(\cdot,t)\|_{L^\infty(0,\infty)}=O(t^{-1}),
\]
the constants depending only on $\|u_0\|_{L^1}$ and $R_0$. As in the proof of
Theorem~\ref{is}, $0<c\le w,w_M\le C$ uniformly, and
\[
u-u_M=\frac{w_x-(w_M)_x}{w}-\frac{(w_M)_x}{w\,w_M}\,(w-w_M),
\]
where $\|w(\cdot,t)-w_M(\cdot,t)\|_{L^\infty}\le\|w_x-(w_M)_x\|_{L^1}=O(t^{-1/2})$ and
$|(w_M)_x(x,t)|\le Ct^{-1/2}$. Hence $|u-u_M|\le C t^{-1}$ for $0\le x\le R\sqrt t$, and
multiplying by $t^{1/2}$ gives \eqref{eq:hlrate-statement}.
\end{proof}

\begin{remark}

The exponent in \eqref{eq:hlrate-statement} is not claimed to be optimal;
numerically one observes a faster rate on compact sets in similarity variables.
Theorem~\ref{thm:halflineCk} below shows that the same $t^{-1/2}$ bound holds in
every local $C^k$ norm, which is the quantitative form used later.
\end{remark}

\begin{teo}\label{thm:halflineCk}
Under the hypotheses of Theorem~\ref{thm:halflinerate}, for every $k\in\mathbb N_0$ and
every $R>0$, there is $C_k>0$ such that
\[
\|v^\infty(\cdot,t)-f_M\|_{C^k([0,R])}\le C_k t^{-1/2},
\qquad t\ge1 .
\]
\end{teo}

\begin{proof}
Use the notation of Theorem~\ref{thm:halflinerate} and put
$d=w-w_M$.  Its initial datum $d_0=w_0-e^M$ is supported in $[0,R_0]$;
its odd extension belongs to $L^1(\mathbb R)$.  Direct differentiation of the
whole-line heat kernel therefore gives, for every $m\ge0$,
\[
 \|\partial_x^m d(\cdot,t)\|_{L^\infty(0,\infty)}
 \le C_m t^{-(m+1)/2},\qquad t\ge1.
\]
The same kernel calculation gives
$|\partial_x^m w|+|\partial_x^m w_M|\le C_m t^{-m/2}$ for $m\ge1$,
while $w,w_M$ and their reciprocals are uniformly bounded.  Since
\[
 u-u_M=\frac{d_x}{w}-\frac{(w_M)_x}{ww_M}\,d,
\]
the Leibniz rule and the derivative formula for a reciprocal show
\[
 \|\partial_x^k(u-u_M)(\cdot,t)\|_{L^\infty(0,\infty)}
 \le C_k t^{-(k+2)/2}.
\]
Finally,
$\partial_\xi^k(v^\infty-f_M)
=t^{(k+1)/2}\partial_x^k(u-u_M)$ at $x=\xi\sqrt t$.
Multiplication by $t^{(k+1)/2}$ proves the asserted $t^{-1/2}$ rate, uniformly
on $[0,R]$ (indeed, on the entire half-line).
\end{proof}

\section{The diffusive scale: transition profiles and sharp rates}\label{sm}

Throughout this section $u^{L}$ and $u^{\infty}$ denote the solutions of \eqref{eb} and
\eqref{ei} with the same initial datum $u_{0}\in L^{1}(0,\infty)$, assumed to satisfy
\begin{equation}\label{eq:hyp}
\operatorname{supp}u_{0}\subset[0,R_{0}],
\qquad
M:=\int_{0}^{\infty}u_{0}(x)\,dx ,
\end{equation}
and $w^{L}=\exp\big(\int_{0}^{x}u^{L}\big)$, $w^{\infty}=\exp\big(\int_{0}^{x}u^{\infty}\big)$
are their Hopf--Cole transforms. Thus $w^{L}$ solves the heat equation on $(0,L)$ with
$w^{L}(0,t)=1$ and $w^{L}(L,t)=e^{M}$, while $w^{\infty}$ solves the heat equation on
$(0,\infty)$ with $w^{\infty}(0,t)=1$; both have initial datum
$w_{0}=\exp\big(\int_{0}^{x}u_{0}\big)$, which equals $e^{M}$ on $[R_{0},\infty)$ and
satisfies
\begin{equation}\label{eq:w0bounds}
e^{-\|u_{0}\|_{L^{1}}}\le w_{0}\le e^{\|u_{0}\|_{L^{1}}},
\qquad |M|\le\|u_{0}\|_{L^{1}} .
\end{equation}

Unless stated otherwise, constants in this section may depend on
$\|u_0\|_{L^1}$, $R_0$, the observation radius $R$, the derivative order $k$,
and an indicated loss $\varepsilon$, but never on $L$, $t$, or
$c=t/L^2$.  Dependence on a fixed value of $c$ is displayed explicitly in the
transition limit of Subsection~\ref{ss:transition}.  This convention is
important only for the estimates; all limiting statements specify which
parameters are held fixed.

At diffusive times $t=cL^{2}$ we use the similarity variable $\xi=x/\sqrt t$ and the
rescaled profiles
\begin{equation}\label{eq:rescaled}
v^{L}(\xi,t)=\sqrt t\,u^{L}(\xi\sqrt t,t),
\qquad
v^{\infty}(\xi,t)=\sqrt t\,u^{\infty}(\xi\sqrt t,t),
\end{equation}
the first being defined for $0\le\xi\le L/\sqrt t=c^{-1/2}$.

Section \ref{ss:transition} produces an explicit one-parameter family
$\{\Phi_{c}\}_{c>0}$, the $L\to\infty$ limit of $v^{L}(\cdot,cL^{2})$ for each fixed
$c>0$. In similarity coordinates it degenerates to the self-similar profile $f_{M}$ as
$c\to0$, while after the domain-scale rescaling $y=\xi\sqrt c$ it degenerates to the
unit-interval stationary profile $\mathcal U_{M}$ as $c\to\infty$. Section \ref{ss:sharp}
determines the sharp exponential scale of the first degeneration, with a matching lower
bound. Section
\ref{ss:uniform} establishes the corresponding uniform-in-$L$ comparison of $v^{L}$ with
$v^{\infty}$ in $C^{k}$, with the same exponent. Section \ref{ss:regimes} describes the
dynamics at all time scales.

\subsection{The transition family}\label{ss:transition}

Fix $c>0$ and set
\begin{equation}\label{eq:hatw}
y=\frac{x}{L},\qquad \tau=\frac{t}{L^{2}},\qquad
\hat w^{L}(y,\tau):=w^{L}(yL,\tau L^{2}),
\end{equation}
so that $\hat w^{L}$ solves
\begin{equation}\label{eq:hatwpb}
\begin{cases}
\hat w_{\tau}=\hat w_{yy}, & (y,\tau)\in(0,1)\times(0,\infty),\\[1mm]
\hat w(0,\tau)=1,\quad \hat w(1,\tau)=e^{M}, & \tau>0,\\[1mm]
\hat w(y,0)=w_{0}(yL), & y\in(0,1).
\end{cases}
\end{equation}
By \eqref{eq:hyp} the initial datum of \eqref{eq:hatwpb} equals $e^{M}$ on $[R_{0}/L,1]$;
as $L\to\infty$ it converges to the constant $e^{M}$, the discrepancy being confined to a
layer of width $R_{0}/L$ at the origin. Let $\varphi$ solve the limit problem
\begin{equation}\label{eq:phipb}
\begin{cases}
\varphi_{\tau}=\varphi_{yy}, & (y,\tau)\in(0,1)\times(0,\infty),\\[1mm]
\varphi(0,\tau)=0,\quad\varphi(1,\tau)=1, & \tau>0,\\[1mm]
\varphi(y,0)=1, & y\in(0,1),
\end{cases}
\end{equation}
and set
\begin{equation}\label{eq:Psi}
\Psi(y,\tau):=1+(e^{M}-1)\varphi(y,\tau),
\end{equation}
which solves the heat equation on $(0,1)$ with $\Psi(0,\tau)=1$, $\Psi(1,\tau)=e^{M}$ and
$\Psi(\cdot,0)\equiv e^{M}$. Since $\varphi-y$ solves the heat equation with homogeneous
Dirichlet conditions and initial datum $1-y$, whose coefficients in the basis
$\{\sin(n\pi y)\}_{n\ge1}$ are $2/(n\pi)$,
\begin{equation}\label{eq:phiseries}
\varphi(y,\tau)=y+\frac{2}{\pi}\sum_{n\ge1}\frac{\sin(n\pi y)}{n}\,e^{-n^{2}\pi^{2}\tau},
\qquad
\varphi_{y}(y,\tau)=1+2\sum_{n\ge1}\cos(n\pi y)\,e^{-n^{2}\pi^{2}\tau},
\end{equation}
the series converging in $C^{\infty}([0,1])$ for every $\tau>0$. By the maximum principle
$0\le\varphi\le1$, whence
\begin{equation}\label{eq:Psibounds}
0<\min(1,e^{M})\le\Psi\le\max(1,e^{M}).
\end{equation}

\begin{lemma}\label{lem:hatw}
Assume \eqref{eq:hyp}. For every $j\in\mathbb N_{0}$ there is
$C_{j}=C_{j}(\|u_{0}\|_{L^{1}},R_{0},j)$ such that, for all $\tau>0$ and $L\ge R_{0}$,
\begin{equation}\label{eq:hatwest}
\big\|\partial_{y}^{j}\big(\hat w^{L}(\cdot,\tau)-\Psi(\cdot,\tau)\big)\big\|_{L^{\infty}(0,1)}
\le
\frac{C_{j}}{L}\,\min\{\tau,1\}^{-(j+1)/2}.
\end{equation}
\end{lemma}

\begin{proof}
Set $\eta:=\hat w^{L}-\Psi$. Both functions solve the heat equation on $(0,1)$ with the
same boundary data, so $\eta$ solves it with homogeneous Dirichlet conditions and initial
datum $\eta_{0}(y)=w_{0}(yL)-e^{M}$. By \eqref{eq:hyp} and \eqref{eq:w0bounds}, $\eta_{0}$
is supported in $[0,R_{0}/L]$ and bounded by $2e^{\|u_{0}\|_{L^{1}}}$, so
\begin{equation}\label{eq:eta0L1}
\|\eta_{0}\|_{L^{1}(0,1)}\le\frac{A}{L},
\qquad A:=2e^{\|u_{0}\|_{L^{1}}}R_{0}.
\end{equation}
Let $G_{1}$ be the Dirichlet heat kernel of $(0,1)$. For $0<\tau\le1$ the method of images
represents $G_{1}$ as an absolutely convergent signed sum of Gaussians of variance
$2\tau$, and termwise differentiation gives
$|\partial_{y}^{j}G_{1}(\tau,y,\zeta)|\le C_{j}\tau^{-(j+1)/2}$ uniformly on $(0,1)^{2}$.
For $\tau\ge1$ the spectral representation
$G_{1}(\tau,y,\zeta)=2\sum_{n\ge1}e^{-n^{2}\pi^{2}\tau}\sin(n\pi y)\sin(n\pi\zeta)$ gives
$|\partial_{y}^{j}G_{1}(\tau,y,\zeta)|\le C_{j}e^{-\pi^{2}\tau}\le C_{j}$. Since
$\eta(y,\tau)=\int_{0}^{1}G_{1}(\tau,y,\zeta)\eta_{0}(\zeta)\,d\zeta$, estimate
\eqref{eq:hatwest} follows from \eqref{eq:eta0L1}.
\end{proof}

\begin{definition}\label{def:Phic}
For $c>0$ and $0\le\xi\le c^{-1/2}$,
\begin{equation}\label{eq:Phic}
\Phi_{c}(\xi):=\sqrt c\,
\frac{(e^{M}-1)\,\varphi_{y}\big(\xi\sqrt c,c\big)}
     {1+(e^{M}-1)\,\varphi\big(\xi\sqrt c,c\big)}
=\sqrt c\,\frac{\Psi_{y}(\xi\sqrt c,c)}{\Psi(\xi\sqrt c,c)} .
\end{equation}
\end{definition}

By \eqref{eq:Psibounds} the denominator is bounded away from zero for every real $M$, so
$\Phi_{c}$ is well defined and smooth on the closed interval $[0,c^{-1/2}]$.

\begin{teo}\label{thm:Phi}
Assume \eqref{eq:hyp} and fix $c>0$, $k\in\mathbb N_{0}$ and $R\in(0,c^{-1/2})$. Then there
is $C=C(\|u_{0}\|_{L^{1}},R_{0},R,k,c)$ such that, for $t=cL^{2}$ and all $L\ge R_{0}$,
\begin{equation}\label{eq:PhiConv}
\big\|v^{L}(\cdot,cL^{2})-\Phi_{c}\big\|_{C^{k}([0,R])}\le\frac{C}{L}.
\end{equation}
In particular $v^{L}(\cdot,cL^{2})\to\Phi_{c}$ in $C^{k}_{\mathrm{loc}}\big([0,c^{-1/2})\big)$
as $L\to\infty$, for every fixed $c>0$.
\end{teo}

\begin{proof}
Since $u^{L}=w^{L}_{x}/w^{L}$ and $\partial_{x}=L^{-1}\partial_{y}$ under \eqref{eq:hatw},
at $t=cL^{2}$ and $x=\xi\sqrt t$, so that $y=x/L=\xi\sqrt c$, and using $\sqrt t/L=\sqrt c$,
\begin{equation}\label{eq:vLhat}
v^{L}(\xi,cL^{2})
=\sqrt t\,\frac{L^{-1}\hat w^{L}_{y}(\xi\sqrt c,c)}{\hat w^{L}(\xi\sqrt c,c)}
=\sqrt c\,\frac{\hat w^{L}_{y}(\xi\sqrt c,c)}{\hat w^{L}(\xi\sqrt c,c)} .
\end{equation}
Comparing with \eqref{eq:Phic}, the claim reduces to estimating
$\partial_{y}^{j}\big(\hat w^{L}_{y}/\hat w^{L}-\Psi_{y}/\Psi\big)$ at $\tau=c$ on
$[0,R\sqrt c]\subset[0,1]$.

Put $\eta=\hat w^{L}-\Psi$. The maximum principle applied to \eqref{eq:hatwpb}, together
with \eqref{eq:w0bounds}, gives
\begin{equation}\label{eq:hatwlower}
m_{0}:=e^{-\|u_{0}\|_{L^{1}}}\le\hat w^{L}\le e^{\|u_{0}\|_{L^{1}}} .
\end{equation}
From \eqref{eq:phiseries}, $|\partial_{y}^{i}\Psi(\cdot,c)|\le K_{i}(c)$ on $[0,1]$, while
Lemma \ref{lem:hatw} gives $|\partial_{y}^{i}\eta(\cdot,c)|\le C_{i}(c)/L$ and hence
$|\partial_{y}^{i}\hat w^{L}(\cdot,c)|\le K_{i}(c)+C_{i}(c)/L$. Since $\hat w^{L}\ge m_{0}>0$
and $\Psi\ge\min(1,e^{M})>0$, Fa\`a di Bruno yields
\begin{equation}\label{eq:invbounds}
\Big|\partial_{y}^{b}\Big(\frac{1}{\hat w^{L}}\Big)\Big|\le K'_{b}(c),
\qquad
\Big|\partial_{y}^{b}\Big(\frac{1}{\Psi}\Big)\Big|\le K'_{b}(c),
\qquad b\le k+1,\ L\ge R_{0}.
\end{equation}
Writing
\[
\frac{\hat w^{L}_{y}}{\hat w^{L}}-\frac{\Psi_{y}}{\Psi}
=\frac{\eta_{y}}{\hat w^{L}}-\frac{\Psi_{y}}{\hat w^{L}\Psi}\,\eta
\]
and differentiating $j\le k$ times, the Leibniz rule produces finitely many terms, each
carrying exactly one factor $\partial_{y}^{a}\eta$ with $a\le j+1$, the remaining factors
being derivatives of $\Psi$, $1/\hat w^{L}$ and $1/\Psi$ of order at most $j+1$. By
\eqref{eq:invbounds} and Lemma \ref{lem:hatw},
\[
\Big|\partial_{y}^{j}\Big(\frac{\hat w^{L}_{y}}{\hat w^{L}}-\frac{\Psi_{y}}{\Psi}\Big)(y,c)\Big|
\le\frac{C_{j}(c)}{L},
\qquad y\in[0,1].
\]
Since $\partial_{\xi}^{j}=c^{j/2}\partial_{y}^{j}$ along $y=\xi\sqrt c$, \eqref{eq:vLhat}
gives \eqref{eq:PhiConv}.
\end{proof}

\begin{prop}[$c\to\infty$: the stationary profile]\label{prop:cinfty}
Uniformly for $y\in[0,1]$, and at rate $e^{-\pi^{2}c}$,
\begin{equation}\label{eq:Phicinfty}
\frac{1}{\sqrt c}\,\Phi_{c}\Big(\frac{y}{\sqrt c}\Big)
\longrightarrow
\frac{e^{M}-1}{1+(e^{M}-1)y}=\mathcal U_{M}(y)
\qquad\text{as }c\to\infty .
\end{equation}
For comparison, for fixed $L$ and $t\to\infty$,
\begin{equation}\label{eq:UMlimit}
u^{L}(x,t)\longrightarrow \frac{e^{M}-1}{L+(e^{M}-1)x}=U_{M}^{L}(x)
\qquad\text{uniformly on }[0,L].
\end{equation}
\end{prop}

\begin{proof}
For $\tau\ge1$ the series in \eqref{eq:phiseries} are dominated by $Ce^{-\pi^{2}\tau}$, so
$\varphi(\cdot,\tau)\to y$ and $\varphi_{y}(\cdot,\tau)\to1$ uniformly on $[0,1]$ at that
rate, whence $\Psi(y,\tau)\to1+(e^{M}-1)y$ and $\Psi_{y}(y,\tau)\to e^{M}-1$. By
\eqref{eq:Phic}, $c^{-1/2}\Phi_{c}(y/\sqrt c)=\Psi_{y}(y,c)/\Psi(y,c)$, and the denominator
is bounded below by \eqref{eq:Psibounds}; this is \eqref{eq:Phicinfty}.

For \eqref{eq:UMlimit}, fix $L$ and let $\tau=t/L^{2}\to\infty$. The difference
$\hat w^{L}-[1+(e^M-1)y]$ solves the homogeneous Dirichlet heat equation on $(0,1)$ with
bounded initial datum; parabolic smoothing on $\tau\ge1$ together with the spectral gap
gives $C^{1}([0,1])$ decay at the rate $e^{-\pi^{2}\tau}$. Thus
$\hat w^{L}(\cdot,\tau)\to1+(e^{M}-1)\,\cdot$ in $C^{1}([0,1])$, that is
$w^{L}(x,t)\to1+(e^{M}-1)x/L$ together with $w^{L}_{x}(x,t)\to(e^{M}-1)/L$, uniformly on
$[0,L]$. Since $w^{L}\ge m_{0}>0$ by \eqref{eq:hatwlower},
\[
u^{L}=\frac{w^{L}_{x}}{w^{L}}\longrightarrow
\frac{(e^{M}-1)/L}{1+(e^{M}-1)x/L}=\frac{e^{M}-1}{L+(e^{M}-1)x}=U_{M}^{L}(x),
\]
the stationary profile \eqref{UM}.
\end{proof}

\begin{remark}\label{rem:scaling}
The rescaling in \eqref{eq:Phicinfty} is essential: $\Phi_{c}$ is a profile in the
diffusive variable $\xi=x/\sqrt t$, whereas the stationary state lives on the scale of the
domain, $y=x/L=\xi\sqrt c$. A fixed positive $\xi$ eventually lies outside the domain of
$\Phi_c$ as $c\to\infty$. The meaningful statement is
$c^{-1/2}\Phi_c(y/\sqrt c)\to\mathcal U_M(y)$ for fixed $y\in[0,1]$; at $y=0$ this also gives
$\Phi_c(0)\sim\sqrt c\,(e^M-1)$.
\end{remark}

 By uniqueness of the steady state in each mass class, $U_M^L$ is the only equilibrium
a mass-$M$ trajectory can approach. The family $\{\Phi_c\}_{c>0}$ therefore links $f_M$
and $\mathcal U_M$ through limits in their respective natural coordinates; its existence
identifies $t\sim L^2$ as the transition scale.

\subsection{The sharp rate}\label{ss:sharp}

\begin{lemma}\label{lem:poisson}
For every $\tau>0$ and $y\in\mathbb R$,
\begin{equation}\label{eq:poisson}
1+2\sum_{n\ge1}\cos(n\pi y)\,e^{-n^{2}\pi^{2}\tau}
=\frac{1}{\sqrt{\pi\tau}}\sum_{k\in\mathbb Z}
\exp\Big(-\frac{(y-2k)^{2}}{4\tau}\Big),
\end{equation}
both sides converging absolutely. In particular \eqref{eq:poisson} represents
$\varphi_{y}(y,\tau)$ for $y\in[0,1]$.
\end{lemma}

\begin{proof}
Let $g(s)=(4\pi\tau)^{-1/2}e^{-s^{2}/(4\tau)}$, so that
$\hat g(\lambda)=\int_{\mathbb R}g(s)e^{-i\lambda s}\,ds=e^{-\lambda^{2}\tau}$. Both $g$ and
$\hat g$ are Schwartz, so Poisson summation over the lattice $2\mathbb Z$ applies:
\[
\sum_{k\in\mathbb Z}g(y-2k)=\frac12\sum_{n\in\mathbb Z}\hat g(n\pi)\,e^{in\pi y}
=\frac12\sum_{n\in\mathbb Z}e^{-n^{2}\pi^{2}\tau}e^{in\pi y}.
\]
Multiplying by $2$ and pairing $n$ with $-n$ gives \eqref{eq:poisson}.
\end{proof}

The right-hand side of \eqref{eq:poisson} is the method-of-images representation generated
by the reflections of the source across the boundaries $y\in2\mathbb Z$. Its $k=0$ term
produces the half-line profile $f_{M}$; the term $k=1$, namely
$(\pi\tau)^{-1/2}e^{-(y-2)^{2}/(4\tau)}$, is the image charge at $y=2$, that is at $x=2L$,
and is the leading correction. In the variables of the problem its exponent is
$-(2L-x)^{2}/4t$: the Gaussian cost of the path from the origin out to the boundary at
$x=L$ and back to $x$. This is the mechanism by which the boundary is felt, and it fixes
the exponent.

The estimates below are needed for every order of differentiation, so we first record
the elementary bound on the derivatives of a Gaussian that converts each $y$-derivative
of an image term into a power of $c$.

\begin{lemma}\label{lem:gaussder}
Let $m\in\mathbb N_{0}$ and $\Lambda>0$. There is $C_{m}=C_{m}(\Lambda)$ such that, for
every $0<c\le1$, every $a\in\mathbb R$ and every $y$ with $|y-a|\le\Lambda$,
\begin{equation}\label{eq:gaussder}
\Big|\partial_{y}^{m}e^{-(y-a)^{2}/(4c)}\Big|
\;\le\;C_{m}\,c^{-m}\,e^{-(y-a)^{2}/(4c)} .
\end{equation}
\end{lemma}

\begin{proof}
Write $s=(y-a)/(2\sqrt c)$, so that $\partial_{y}=(2\sqrt c)^{-1}\partial_{s}$ and
$e^{-(y-a)^{2}/(4c)}=e^{-s^{2}}$. By the definition of the Hermite polynomials,
$\partial_{s}^{m}e^{-s^{2}}=(-1)^{m}H_{m}(s)e^{-s^{2}}$, whence
\[
\partial_{y}^{m}e^{-(y-a)^{2}/(4c)}
=(-1)^{m}(2\sqrt c)^{-m}H_{m}(s)\,e^{-s^{2}} .
\]
Since $H_{m}$ is a polynomial of degree $m$, $|H_{m}(s)|\le C_{m}(1+|s|)^{m}$. The
hypothesis $|y-a|\le\Lambda$ gives $|s|\le\Lambda/(2\sqrt c)$, so
$(1+|s|)^{m}\le C_{m}(1+\Lambda)^{m}c^{-m/2}$ for $c\le1$. Combining the two factors
yields $(2\sqrt c)^{-m}(1+|s|)^{m}\le C_{m}c^{-m}$, which is \eqref{eq:gaussder}.
\end{proof}

\begin{teo}\label{thm:sharp}
Assume \eqref{eq:hyp} with $M\ne0$ and fix $R>0$. There exist
$c_{1}=c_{1}(R,M)>0$ and $\kappa_{-}=\kappa_{-}(M,R)>0$ such that, for $0<c\le c_{1}$:
\begin{enumerate}
\item[(i)] for every $k\in\mathbb N_{0}$ there exists
$\kappa_{+,k}=\kappa_{+,k}(M,R)>0$ such that
\begin{equation}\label{eq:sharpupper}
\big\|\Phi_{c}-f_{M}\big\|_{C^{k}([0,R])}
\le
\kappa_{+,k}\,c^{-k/2}\exp\Big(-\frac1c+\frac{R}{\sqrt c}\Big),
\end{equation}
and hence, for every $\varepsilon>0$, with $C_{k,\varepsilon}$ independent of $c$,
\begin{equation}\label{eq:sharpupper2}
\big\|\Phi_{c}-f_{M}\big\|_{C^{k}([0,R])}\le C_{k,\varepsilon}\,e^{-(1-\varepsilon)/c};
\end{equation}
\item[(ii)] at $\xi=0$ the difference has the exact asymptotics
\begin{equation}\label{eq:sharpasymp}
\big|\Phi_{c}(0)-f_{M}(0)\big|
=\frac{2|e^{M}-1|}{\sqrt\pi}\,e^{-1/c}\bigl(1+O(e^{-3/c})\bigr),
\qquad c\to0^{+};
\end{equation}
in particular
\begin{equation}\label{eq:sharplower}
\big|\Phi_{c}(0)-f_{M}(0)\big|\ \ge\ \kappa_{-}\,e^{-1/c},
\end{equation}
and
\begin{equation}\label{eq:gammaone}
\lim_{c\to0^{+}}\ -c\,\ln\big|\Phi_{c}(0)-f_{M}(0)\big|=1 .
\end{equation}
\end{enumerate}
The exponent $\gamma=1$ in \eqref{eq:sharpupper2} is therefore attained, and no larger
exponent is admissible.
\end{teo}

\begin{proof}
Fix $c_{1}\le\min \{1,1/(4R^{2})\}$, so that $y:=\xi\sqrt c\in[0,\tfrac12]$ for $\xi\in[0,R]$.

\emph{(i)} By Lemma \ref{lem:poisson},
\begin{equation}\label{eq:split1}
\varphi_{y}(y,c)=\frac{1}{\sqrt{\pi c}}e^{-y^{2}/(4c)}+E_{1}(c,y),
\qquad
E_{1}(c,y)=\frac{1}{\sqrt{\pi c}}\sum_{\ell\ne0}e^{-(y-2\ell)^{2}/(4c)} .
\end{equation}
For $0\le y\le\tfrac12$ and $\ell\ne0$ we have $|y-2\ell|\ge2|\ell|-y$, so
$(y-2\ell)^{2}\ge(2|\ell|-y)^{2}$, and the smallest of these values is $(2-y)^{2}$, attained at
$\ell=1$. Moreover, for $|\ell|\ge1$ and $y\le\tfrac12$,
\[
(2|\ell|-y)^{2}-(2-y)^{2}=4(|\ell|-1)\big(|\ell|+1-y\big)\ge 2(|\ell|-1),
\]
so the terms of $E_{1}$ are dominated by $e^{-(2-y)^{2}/(4c)}$ times a summable geometric
series of ratio $e^{-1/(2c)}$. Hence
\begin{equation}\label{eq:E1}
0<E_{1}(c,y)\le\frac{C}{\sqrt c}\,e^{-(2-y)^{2}/(4c)}
=\frac{C}{\sqrt c}\exp\Big(-\frac1c+\frac{\xi}{\sqrt c}-\frac{\xi^{2}}{4}\Big),
\end{equation}
the last equality expanding $(2-y)^{2}/(4c)=1/c-y/c+y^{2}/(4c)$ and substituting
$y=\xi\sqrt c$. Integrating \eqref{eq:split1} from $0$, using $\varphi(0,c)=0$ and
$\int_{0}^{y}(\pi c)^{-1/2}e^{-s^{2}/(4c)}\,ds=\operatorname{erf}\big(y/(2\sqrt c)\big)$,
\begin{equation}\label{eq:split0}
\varphi(y,c)=\operatorname{erf}\Big(\frac{y}{2\sqrt c}\Big)+E_{0}(c,y),
\qquad
E_{0}(c,y)=\int_{0}^{y}E_{1}(c,s)\,ds .
\end{equation}
By \eqref{eq:E1} the integrand is positive and, on $[0,\tfrac12]$, bounded by its value at
$s=y$, since $s\mapsto e^{-(2-s)^{2}/(4c)}$ is increasing there; therefore
\begin{equation}\label{eq:E0}
0<E_{0}(c,y)\le y\cdot\frac{C}{\sqrt c}e^{-(2-y)^{2}/(4c)}
\le C\exp\Big(-\frac1c+\frac{\xi}{\sqrt c}-\frac{\xi^{2}}{4}\Big),
\end{equation}
absorbing $y/\sqrt c=\xi\le R$ into the constant.

Substituting $y=\xi\sqrt c$ into \eqref{eq:split1} and \eqref{eq:split0}, and using
$\operatorname{erf}(\xi/2)=\int_{0}^{\xi}b$,
\[
\sqrt c\,\varphi_{y}(\xi\sqrt c,c)=\frac{1}{\sqrt\pi}e^{-\xi^{2}/4}+\sqrt c\,E_{1},
\qquad
\varphi(\xi\sqrt c,c)=\int_{0}^{\xi}b+E_{0}.
\]
Shrinking $c_{1}$ so that $|e^{M}-1|\sup_{[0,R]}|E_{0}|\le\tfrac12\min(1,e^{M})$, both
denominators in $\Phi_{c}(\xi)$ and $f_{M}(\xi)$ are bounded below by
$\tfrac12\min(1,e^{M})>0$, and the quotient rule gives
\[
|\Phi_{c}(\xi)-f_{M}(\xi)|\le C\big(\sqrt c\,|E_{1}|+|E_{0}|\big)
\le\kappa_{+}\exp\Big(-\frac1c+\frac{R}{\sqrt c}\Big),
\]
which is \eqref{eq:sharpupper} for $k=0$.

For $k\ge1$ we first estimate the derivatives of the two remainders. Fix $m\in\mathbb
N_{0}$. Differentiating \eqref{eq:split1} termwise and applying Lemma \ref{lem:gaussder}
with $a=2\ell$ and $\Lambda=2|\ell|+\tfrac12$, which is legitimate since $0\le y\le\tfrac12$,
\[
\big|\partial_{y}^{m}E_{1}(c,y)\big|
\le\frac{C_{m}}{\sqrt{\pi c}}\,c^{-m}\sum_{\ell\ne0}(1+2|\ell|)^{m}\,e^{-(y-2\ell)^{2}/(4c)} .
\]
By the two inequalities established above, $(y-2\ell)^{2}\ge(2-y)^{2}+2(|\ell|-1)$ for
$\ell\ne0$ and $y\le\tfrac12$, so the sum is bounded by
$e^{-(2-y)^{2}/(4c)}\sum_{\ell\ne0}(1+2|\ell|)^{m}e^{-(|\ell|-1)/(2c)}$, and the last series
converges to a constant depending only on $m$, uniformly for $c\le c_{1}$. Hence
\begin{equation}\label{eq:E1der}
\big|\partial_{y}^{m}E_{1}(c,y)\big|\le C_{m}\,c^{-m-1/2}\,e^{-(2-y)^{2}/(4c)},
\qquad 0\le y\le\tfrac12 .
\end{equation}
Since $E_{0}(c,\cdot)$ is the primitive of $E_{1}(c,\cdot)$ vanishing at $y=0$, one has
$\partial_{y}^{m}E_{0}=\partial_{y}^{m-1}E_{1}$ for $m\ge1$, so \eqref{eq:E1der} and
\eqref{eq:E0} give
\begin{equation}\label{eq:E0der}
\big|\partial_{y}^{m}E_{0}(c,y)\big|\le C_{m}\,c^{-m+1/2}\,e^{-(2-y)^{2}/(4c)},
\qquad m\ge1 .
\end{equation}
Along $y=\xi\sqrt c$ one has $\partial_{\xi}^{m}=c^{m/2}\partial_{y}^{m}$, so
\eqref{eq:E1der}, \eqref{eq:E0der} and \eqref{eq:E0} become
\begin{equation}\label{eq:Eder-xi}
\big|\partial_{\xi}^{m}\big(\sqrt c\,E_{1}\big)\big|\le C_{m}\,c^{-m/2}\,\mathcal E(c,\xi),
\qquad
\big|\partial_{\xi}^{m}E_{0}\big|\le C_{m}\,c^{-m/2}\,\mathcal E(c,\xi),
\end{equation}
for every $m\in\mathbb N_{0}$, where
$\mathcal E(c,\xi):=\exp\big(-\tfrac1c+\tfrac{\xi}{\sqrt c}-\tfrac{\xi^{2}}{4}\big)$ is
the common exponential factor of \eqref{eq:E1} and \eqref{eq:E0}.

It remains to differentiate the quotient. Write $A=e^{M}-1$, $b(\xi)=\pi^{-1/2}e^{-\xi^{2}/4}$
and $B(\xi)=\int_{0}^{\xi}b$, so that
$f_{M}=Ab/(1+AB)$ while, by the display following \eqref{eq:E0},
$\Phi_{c}=A\,(b+\sqrt c\,E_{1})/\big(1+A(B+E_{0})\big)$. Subtracting,
\begin{equation}\label{eq:quotsplit}
\Phi_{c}-f_{M}
=\frac{A\,\sqrt c\,E_{1}}{1+A(B+E_{0})}
-\frac{A^{2}\,b\,E_{0}}{(1+AB)\big(1+A(B+E_{0})\big)} .
\end{equation}
Both denominators are bounded below by $\tfrac12\min(1,e^{M})>0$, by the choice of
$c_{1}$ made above; the derivatives of $b$ and $B$ are bounded on $[0,R]$ independently
of $c$; and by \eqref{eq:Eder-xi} together with Fa\`a di Bruno the derivatives of
$1/(1+A(B+E_{0}))$ up to order $k$ are bounded independently of $c$ as well, since every
factor $\partial_{\xi}^{m}E_{0}$ carries $\mathcal E(c,\xi)\le e^{-1/c+R/\sqrt c}$, which
tends to $0$ as $c\to0^{+}$. Differentiating \eqref{eq:quotsplit} $m\le k$ times, the
Leibniz rule therefore produces finitely many terms, each carrying exactly one factor
$\partial_{\xi}^{a}(\sqrt c\,E_{1})$ or $\partial_{\xi}^{a}E_{0}$ with $a\le m$, the
remaining factors being bounded uniformly in $c$. By \eqref{eq:Eder-xi} each such term is
bounded by $C_{k}c^{-a/2}\mathcal E(c,\xi)\le C_{k}c^{-k/2}\mathcal E(c,\xi)$, and since
$\mathcal E(c,\xi)\le\exp(-1/c+R/\sqrt c)$ on $[0,R]$, summing gives
\eqref{eq:sharpupper}. Finally \eqref{eq:sharpupper2} follows from
$c^{-k/2}e^{R/\sqrt c}\le C_{k,\varepsilon}e^{\varepsilon/c}$.

\emph{(ii)} At $\xi=0$, i.e. $y=0$, we have $\varphi(0,c)=0$ and $\operatorname{erf}(0)=0$, so
$E_{0}(c,0)=0$ and the denominators of $\Phi_{c}(0)$ and $f_{M}(0)$ both equal $1$.
Consequently
\[
\Phi_{c}(0)-f_{M}(0)
=(e^{M}-1)\Big(\sqrt c\,\varphi_{y}(0,c)-\frac{1}{\sqrt\pi}\Big)
=(e^{M}-1)\,\sqrt c\,E_{1}(c,0),
\]
and, by \eqref{eq:split1},
\[
\sqrt c\,E_{1}(c,0)=\frac{2}{\sqrt\pi}\sum_{\ell\ge1}e^{-\ell^{2}/c}
=\frac{2}{\sqrt\pi}\,e^{-1/c}\Big(1+\sum_{\ell\ge2}e^{-(\ell^{2}-1)/c}\Big)
=\frac{2}{\sqrt\pi}\,e^{-1/c}\big(1+O(e^{-3/c})\big).
\]
All terms are positive, so the sum is bounded below by its first term and no cancellation
occurs. Hence, for $M\ne0$ and $c\le c_{1}$,
\[
|\Phi_{c}(0)-f_{M}(0)|=\frac{2|e^{M}-1|}{\sqrt\pi}\,e^{-1/c}\big(1+O(e^{-3/c})\big)
\ \ge\ \kappa_{-}e^{-1/c},
\]
which is \eqref{eq:sharpasymp} and \eqref{eq:sharplower}; taking $-c\ln$ gives \eqref{eq:gammaone}.
\end{proof}

\begin{cor}[$c\to0$: the self-similar profile]\label{cor:c-to-zero}
For every $R>0$ and $k\in\mathbb N_0$, $\|\Phi_c-f_M\|_{C^k([0,R])}\to0$ as $c\to0^+$, where $f_M$ is the self-similar profile of Proposition \ref{TI2}.
\end{cor}
\begin{proof}
If $M\ne0$, this is immediate from Theorem~\ref{thm:sharp}(i), which
gives the quantitative bound \eqref{eq:sharpupper2}.  If $M=0$, then
$e^M-1=0$, so \eqref{eq:Phic} and the formula of Proposition~\ref{TI2}
give $\Phi_c\equiv f_M\equiv0$ for every $c>0$.
\end{proof}

\begin{remark}\label{rem:fit}
The estimate \eqref{eq:sharpupper} contains the subexponential factor
$\exp(R/\sqrt c)$ (and the leading image term contains
$\exp(-1/c+R/\sqrt c-R^{2}/4)$). Thus
$\ln\|\Phi_{c}-f_{M}\|_{C^{0}([0,R])}$ is not affine in $1/c$: the subexponential factor
$e^{R/\sqrt c}$ flattens the apparent slope over any bounded window of $1/c$. A
one-parameter fit of $\ln E$ against $1/c$ therefore returns an effective slope strictly
between $0$ and $1$, depending on the window and not on the problem. The two-parameter
form $-a/c+bR/\sqrt c$ should be fitted instead, or the range restricted to
$R\sqrt c\ll1$; both are consistent with $a=1$, in accordance with \eqref{eq:gammaone}.
\end{remark}

\subsection{Uniform comparison with the half-line problem}\label{ss:uniform}

\begin{lemma}\label{lem:boundary-tail}
Assume \eqref{eq:hyp}. There is $C=C(\|u_{0}\|_{L^{1}},R_{0})>0$ such that
\begin{equation}\label{eq:tail}
\big|e^{M}-w^{\infty}(L,t)\big|\le C\exp\Big(-\frac{(L-R_{0})^{2}}{4t}\Big)
\end{equation}
for all $L>R_{0}$ and $t>0$.
\end{lemma}

\begin{proof}
Set $h:=w^{\infty}-1$, which solves the heat equation on $(0,\infty)$ with $h(0,t)=0$ and
$h(\cdot,0)=h_{0}:=w_{0}-1$, so that
\[
h(x,t)=\int_{0}^{\infty}G_{D}(t,x,y)h_{0}(y)\,dy,
\qquad
G_{D}(t,x,y)=\frac{1}{\sqrt{4\pi t}}\Big(e^{-\frac{(x-y)^{2}}{4t}}-e^{-\frac{(x+y)^{2}}{4t}}\Big),
\]
$G_{D}$ being the Dirichlet heat kernel of the half-line. By \eqref{eq:hyp},
$h_{0}=e^{M}-1$ on $[R_{0},\infty)$. Since $e^{M}-w^{\infty}(L,t)=(e^{M}-1)-h(L,t)$ and
$(e^{M}-1)-h_{0}(y)=e^{M}-w_{0}(y)$ vanishes for $y\ge R_{0}$,
\begin{equation}\label{eq:split}
e^{M}-w^{\infty}(L,t)
=\int_{0}^{R_{0}}G_{D}(t,L,y)\big(e^{M}-w_{0}(y)\big)\,dy
+(e^{M}-1)\Big(1-\int_{0}^{\infty}G_{D}(t,L,y)\,dy\Big).
\end{equation}
For the first term, \eqref{eq:w0bounds} gives
$|e^{M}-w_{0}(y)|\le e^{|M|}+e^{\|u_{0}\|_{L^{1}}}\le2e^{\|u_{0}\|_{L^{1}}}$, and
$0\le G_{D}(t,L,y)\le(4\pi t)^{-1/2}e^{-(L-y)^{2}/(4t)}$; with $z=L-y$,
\[
\Big|\int_{0}^{R_{0}}G_{D}(t,L,y)\big(e^{M}-w_{0}(y)\big)\,dy\Big|
\le C\int_{L-R_{0}}^{\infty}t^{-1/2}e^{-\frac{z^{2}}{4t}}\,dz .
\]
Substituting $z=2\sqrt t\,s$,
\begin{equation}\label{eq:sharptail}
\int_{a}^{\infty}t^{-1/2}e^{-\frac{z^{2}}{4t}}\,dz
=2\int_{a/(2\sqrt t)}^{\infty}e^{-s^{2}}\,ds
=\sqrt\pi\,\operatorname{erfc}\Big(\frac{a}{2\sqrt t}\Big)
\le\sqrt\pi\,e^{-\frac{a^{2}}{4t}},
\qquad a\ge0,
\end{equation}
by the elementary bound $\operatorname{erfc}(\lambda)\le e^{-\lambda^{2}}$, $\lambda\ge0$.
Taking $a=L-R_{0}$ bounds this term by $Ce^{-(L-R_{0})^{2}/(4t)}$. For the second term of
\eqref{eq:split}, $\int_{0}^{\infty}G_{D}(t,L,y)\,dy=\operatorname{erf}\big(L/(2\sqrt t)\big)$,
so it equals $|e^{M}-1|\operatorname{erfc}\big(L/(2\sqrt t)\big)\le Ce^{-L^{2}/(4t)}
\le Ce^{-(L-R_{0})^{2}/(4t)}$. Adding the two bounds gives \eqref{eq:tail}.
\end{proof}

The next lemma is the quantitative form of the image-charge mechanism identified after
Lemma \ref{lem:poisson}: it estimates $z=w^{L}-w^{\infty}$ at interior points with the
exponent $(2L-x)^{2}/4t$, rather than the exponent $L^{2}/4t$ that the maximum principle
alone would give. The gain is the Gaussian decay accumulated on the return path from
$x=L$ to $x=O(\sqrt t)$, and the underlying identity is the elementary optimization
\eqref{eq:CS} below.

\begin{lemma}\label{lem:image}
Assume \eqref{eq:hyp} and let $z := w^{L} - w^{\infty}$ on $(0,L)\times(0,\infty)$. For every
$\theta\in(0,1)$ there is $C_{\theta}=C_{\theta}(\|u_{0}\|_{L^{1}},R_{0})>0$ such that
\begin{equation}\label{eq:zbound}
|z(x,t)| \;\le\; C_{\theta}\,\exp\!\Big(-\theta\,\frac{(2L-R_{0}-x)^{2}}{4t}\Big)
\end{equation}
for all $0<t\le L^{2}$, all $L\ge 2R_{0}$ and all $0\le x\le \tfrac12 L$.
\end{lemma}

\begin{proof}
Both $w^{L}$ and $w^{\infty}$ solve the heat equation on $(0,L)$ with the same initial datum $w_{0}$
and the same value $1$ at $x=0$; they differ only in the datum at $x=L$. Hence $z$ solves
\begin{equation}\label{eq:zsys}
\begin{cases}
z_{s}=z_{xx}, & (x,s)\in(0,L)\times(0,\infty),\\[2pt]
z(0,s)=0,\quad z(L,s)=\beta(s):=e^{M}-w^{\infty}(L,s), & s>0,\\[2pt]
z(x,0)=0, & x\in(0,L),
\end{cases}
\end{equation}
and, the initial datum being zero, the classical boundary (Duhamel) representation gives
\begin{equation}\label{eq:duhamel}
z(x,t)=\int_{0}^{t}\beta(s)\,N(t-s,x)\,ds,
\qquad
N(r,x):=-\,\partial_{y}G_{(0,L)}(r,x,y)\Big|_{y=L},
\end{equation}
$G_{(0,L)}$ being the Dirichlet heat kernel of $(0,L)$.

\medskip
\emph{Step 1 (the boundary kernel)}.
By the method of images,
\[
G_{(0,L)}(r,x,y)=\frac{1}{\sqrt{4\pi r}}\sum_{k\in\mathbb{Z}}
\Big(e^{-\frac{(x-y+2kL)^{2}}{4r}}-e^{-\frac{(x+y+2kL)^{2}}{4r}}\Big).
\]
Differentiating termwise in $y$ and evaluating at $y=L$, the arguments of the two families become
$x-L+2kL=x+(2k-1)L$ and $x+L+2kL=x+(2k+1)L$; the second family is the first one shifted by one
index, so the two merge into a single sum over the odd multiples of $L$ translated by $x$. Writing
\[
d_{m}:=x+(2m-1)L,\qquad m\in\mathbb{Z},
\]
one obtains
\begin{equation}\label{eq:Nseries}
N(r,x)=\frac{2}{\sqrt{4\pi r}}\sum_{m\in\mathbb{Z}}\Big(-\frac{d_{m}}{2r}\Big)e^{-d_{m}^{2}/(4r)}.
\end{equation}
Only an upper bound on $|N|$ is needed below, so we estimate \eqref{eq:Nseries} term by term.
Assume $0\le x\le\tfrac12 L$; then $|d_{0}|=L-x\in[\tfrac12 L,L]$.

The term $m=1$. Here $d_{1}=x+L$, so $x$ adds to $L$ rather than subtracting from it, and no
gain over $|d_{0}|$ is available: at $x=0$ one has $|d_{1}|=|d_{0}|=L$. (This is forced by
$z(0,s)=0$: the image at $m=1$ is precisely the one cancelling $m=0$ at the origin.) It suffices to
compare the two directly. Since $(L+x)^{2}\ge(L-x)^{2}$ and $L+x\le\tfrac32 L\le 3(L-x)$ for
$x\le\tfrac12 L$,
\begin{equation}\label{eq:m1}
\frac{|d_{1}|}{2r}\,e^{-d_{1}^{2}/(4r)}
=\frac{L+x}{2r}\,e^{-(L+x)^{2}/(4r)}
\le \frac{3(L-x)}{2r}\,e^{-(L-x)^{2}/(4r)},
\end{equation}
which merely triples the $m=0$ term and is absorbed into the constant of \eqref{eq:Nbound} below.

The terms $m=-1$ and $|m|\ge 2$. For these, $|d_{m}|\ge(2|m|-1)L-x\ge(2|m|-\tfrac32)L$, and
since $(L-x)^{2}\le L^{2}$,
\[
d_{m}^{2}-(L-x)^{2}\;\ge\;\Big(2|m|-\tfrac32\Big)^{2}L^{2}-L^{2}\;\ge\;\frac{L^{2}}{2}\Big(|m|-\tfrac12\Big),
\]
as is checked directly for $m=-1$ and for $|m|\ge2$. Hence these terms are dominated by
\[
\frac{|d_{m}|}{2r}\,e^{-d_{m}^{2}/(4r)}
\le\frac{|d_{m}|}{2r}\,e^{-(L-x)^{2}/(4r)}\,e^{-L^{2}(|m|-1/2)/(8r)},
\]
and, using $|d_{m}|\le 2(|m|+1)L$,
\[
\sum_{m\ne 0,1}\frac{|d_{m}|}{2r}\,e^{-L^{2}(|m|-1/2)/(8r)}
\;\le\;\frac{2L}{r}\sum_{j\ge1}(j+1)\,e^{-\lambda(j-1/2)/8},
\qquad \lambda:=\frac{L^{2}}{r}\ge1,
\]
the restriction $r\le L^{2}$ being exactly the range in which \eqref{eq:Nbound} is applied in Step~2,
where $r=t-s\le t\le L^{2}$. The last series is therefore bounded by the convergent
$\sum_{j\ge1}(j+1)e^{-(j-1/2)/8}$, uniformly in $r$ and $L$.

Retaining the term $m=0$, whose size is $\frac{L-x}{2r}e^{-(L-x)^{2}/(4r)}$, adding \eqref{eq:m1},
and using $L/r\le C(L-x)/r$ for $x\le\tfrac12 L$ to absorb the prefactor of the remaining sum, the
triangle inequality gives
\begin{equation}\label{eq:Nbound}
|N(r,x)|\;\le\;\frac{C\,(L-x)}{r^{3/2}}\,\exp\!\Big(-\frac{(L-x)^{2}}{4r}\Big),
\qquad 0<r\le L^{2},\quad 0\le x\le\tfrac12 L,
\end{equation}
with $C$ absolute.

\medskip
\emph{Step 2 (Duhamel and the optimal splitting of time)}.
By Lemma~\ref{lem:boundary-tail}, $|\beta(s)|\le Ce^{-(L-R_{0})^{2}/(4s)}$. Inserting this and
\eqref{eq:Nbound} into \eqref{eq:duhamel},
\begin{equation}\label{eq:zint}
|z(x,t)|\le C(L-x)\int_{0}^{t}\frac{1}{(t-s)^{3/2}}
\exp\!\Big(-\frac{(L-R_{0})^{2}}{4s}-\frac{(L-x)^{2}}{4(t-s)}\Big)ds.
\end{equation}
The exponent is bounded below by the Cauchy--Schwarz inequality: for $a,b\ge0$ and $0<s<t$,
\begin{equation}\label{eq:CS}
\frac{a^{2}}{s}+\frac{b^{2}}{t-s}\;\ge\;\frac{(a+b)^{2}}{t},
\end{equation}
with equality at $s=ta/(a+b)$. Indeed \eqref{eq:CS} is the Cauchy--Schwarz inequality applied to the
vectors $(a/\sqrt{s},\,b/\sqrt{t-s})$ and $(\sqrt{s},\,\sqrt{t-s})$. With $a=L-R_{0}$ and $b=L-x$,
\begin{equation}\label{eq:CSapplied}
\frac{(L-R_{0})^{2}}{4s}+\frac{(L-x)^{2}}{4(t-s)}\;\ge\;\frac{(2L-R_{0}-x)^{2}}{4t},
\qquad 0<s<t.
\end{equation}
This is the assertion that the cheapest Gaussian path from the origin to $x$ which touches $x=L$
costs $(2L-R_{0}-x)^{2}/4t$, the times $s$ and $t-s$ being allocated optimally between the outward
and the return leg.

Fix $\theta\in(0,1)$. Applying \eqref{eq:CSapplied} to the fraction $\theta$ of the exponent and
discarding the fraction $1-\theta$ of its first term,
\[
\exp\!\Big(-\frac{(L-R_{0})^{2}}{4s}-\frac{(L-x)^{2}}{4(t-s)}\Big)
\le \exp\!\Big(-\theta\,\frac{(2L-R_{0}-x)^{2}}{4t}\Big)
\exp\!\Big(-(1-\theta)\,\frac{(L-x)^{2}}{4(t-s)}\Big),
\]
so that, substituting $r=t-s$ and then $r=(L-x)^{2}\sigma$,
\[
\int_{0}^{t}\frac{1}{(t-s)^{3/2}}\exp\!\Big(-(1-\theta)\frac{(L-x)^{2}}{4(t-s)}\Big)ds
\;\le\;\int_{0}^{\infty}\frac{e^{-(1-\theta)(L-x)^{2}/(4r)}}{r^{3/2}}\,dr
=\frac{c_{\theta}}{L-x},
\]
with $c_{\theta}=\int_{0}^{\infty}\sigma^{-3/2}e^{-(1-\theta)/(4\sigma)}\,d\sigma<\infty$ for
$\theta<1$. Combining with \eqref{eq:zint}, the factors $L-x$ cancel and \eqref{eq:zbound} follows.
\end{proof}

\begin{remark}\label{rem:optimality}
The exponent in \eqref{eq:zbound} is optimal: as $L\to\infty$ with $x=\xi\sqrt t$ and
$t=cL^{2}$ fixed, $(2L-R_{0}-x)^{2}/4t\to (2-\xi\sqrt{c})^2/4c$ which equals $1/c$ at $\xi=0$, matching the lower bound of Theorem
\ref{thm:sharp}(ii). It is the exponent of the image charge at $x=2L$, and
\eqref{eq:CS} is the statement that the cheapest Gaussian path from the origin to $x$ that
touches $x=L$ costs $(2L-x)^{2}/4t$, the times $s$ and $t-s$ being allocated optimally
between the outward and return legs.
\end{remark}

\begin{lemma}\label{lem:moderate-c}
Let $z=w^L-w^\infty$.  Fix  $R>0$, and $j\in\mathbb N_0$.
There is $C=C(R,j,\|u_0\|_{L^1},R_0)$ such that, whenever
$0< c<R^{-2}$, $t=cL^2$, and $L>R_0$,
\[ t^{j/2}\|\partial_x^jz(\cdot,t)\|_{L^\infty(0,L)}\le C.
\]
\end{lemma}

\begin{proof}
We bound $w^{\infty}$ and $w^{L}$ separately. Since
$w^{\infty}-1=\int_0^{\infty}\bigl(\Gamma(t,x-y)-\Gamma(t,x+y)\bigr)
(w_0(y)-1)\,dy$ with $\Gamma$ the Gaussian kernel and $|w_0-1|\le
2e^{\|u_0\|_{L^1}}$, the classical bound
$\int_{\mathbb R}|\partial_x^{m}\Gamma(t,s)|\,ds=C_mt^{-m/2}$ gives
$|\partial_x^{m}w^{\infty}(x,t)|\le C_mt^{-m/2}$ for $m\ge1$, all
$x\ge0$ and $t>0$.

For $w^{L}$, pass to domain variables
$\hat w^{L}(y,\tau)=w^{L}(Ly,L^{2}\tau)$ and let $W(y)=1+(e^{M}-1)y$.
Then $\hat w^{L}-W$ solves the homogeneous Dirichlet heat equation on
$(0,1)$ with initial datum bounded by $2e^{\|u_0\|_{L^1}}$ uniformly in
$L$, so the kernel bounds
$\int_0^1|\partial_y^{j}G_1(\tau,y,\zeta)|\,d\zeta\le C_j\min\{\tau,1\}^{-j/2}$
(images for $\tau\le1$, the spectral series for $\tau\ge1$) yield
$\|\partial_y^{j}\hat w^{L}(\cdot,\tau)\|_{L^{\infty}(0,1)}\le
C_j\min\{\tau,1\}^{-j/2}$, with $C_j$ independent of $L$. Since
$\partial_x^{j}w^{L}=L^{-j}\partial_y^{j}\hat w^{L}$ and
$t^{j/2}L^{-j}=c^{j/2}$, this gives
\begin{equation}\label{eq:wL-der}
t^{j/2}\|\partial_x^{j}w^{L}(\cdot,t)\|_{L^{\infty}(0,L)}
\le C_j\,c^{j/2}\min\{c,1\}^{-j/2}\le C_j\max\{1,R^{-j}\},
\qquad 0<c<R^{-2},
\end{equation}

The claim for $j\ge1$ follows by the triangle inequality; for $j=0$ it
is the maximum principle.
\end{proof}

\begin{teo}\label{thm:main-convergence}
Assume \eqref{eq:hyp}. Fix $R>0$, $k\in\mathbb N_{0}$ and $\varepsilon\in(0,1)$, and let
$t=cL^{2}$ with
\[
0<c<\frac{1}{R^{2}} .
\]
Then there exist $C_{k}=C_{k}(\|u_{0}\|_{L^{1}},R_{0},R,k,\varepsilon)>0$, independent of
$c$ and $L$, and $L_{0}=L_{0}(R_{0},R,\varepsilon)$, such that
\begin{equation}\label{ck}
\big\|v^{L}(\cdot,t)-v^{\infty}(\cdot,t)\big\|_{C^{k}([0,R])}
\le
C_{k}\,e^{-(1-\varepsilon)/c},
\qquad L\ge L_{0} .
\end{equation}
If $M\ne0$, the exponential constant $1$ is optimal in the following
uniform sense: no estimate of this form can hold with $1-\varepsilon$ replaced
by a fixed number larger than $1$.
\end{teo}

\begin{proof}
Let $z=w^{L}-w^{\infty}$.

\emph{Step 1 ($L^{\infty}$ bound).} Let $\varepsilon\in(0,1)$ and choose
$\theta=\theta(\varepsilon)\in(0,1)$ and $c_{2}=c_{2}(R,\varepsilon)>0$ as follows. Take
$L_{0}$ so large that $R_{0}/L\le\varepsilon/8$ for $L\ge L_{0}$, and restrict first to
$c\le c_{2}:=\min\{1,\varepsilon^{2}/[256(1+R)^{2}]\}$, so that $R\sqrt c\le\varepsilon/8$. Then, for
$0\le x\le R\sqrt t$ and $t=cL^{2}\le L^{2}$,
\[
\frac{(2L-R_{0}-x)^{2}}{4t}
\ \ge\ \frac{L^{2}\big(2-\varepsilon/8-\varepsilon/8\big)^{2}}{4t}
\ \ge\ \frac{(1-\varepsilon/4)}{c},
\]
using $(2-\varepsilon/4)^{2}/4\ge1-\varepsilon/4$. Choosing
$\theta:=(1-\varepsilon)/(1-\varepsilon/4)\in(0,1)$, Lemma \ref{lem:image} applies, since
$t\le L^{2}$ and $x\le R\sqrt t\le\tfrac12 L$ for $L\ge L_{0}$, and gives
\begin{equation}\label{eq:zsup}
\sup_{0\le x\le R\sqrt t}|z(x,t)|\le C_{\theta}\,e^{-(1-\varepsilon)/c},
\end{equation}
with $C_{\theta}$ independent of $c$ and $L$. For
$c\in[c_{2},1/R^{2})$, Lemma~\ref{lem:moderate-c} gives every derivative bound
needed below, uniformly up to both endpoints.  Since
$e^{-(1-\varepsilon)/c}$ is bounded below by
$e^{-(1-\varepsilon)/c_2}$ in this range, those bounds are absorbed into the
constant in \eqref{ck}.

\emph{Step 2 (derivative estimates).} Fix $0\le x\le R\sqrt t$ and set
\[
\rho:=\delta\sqrt t,
\qquad
\delta:=\min\Big\{\frac14,\ \frac{1-R\sqrt c}{2\sqrt c}\Big\}>0,
\]
which is positive precisely because $c<1/R^{2}$. Using $\sqrt t=\sqrt c\,L$,
\begin{equation}\label{eq:geom}
L-R\sqrt t-\rho=L\big(1-R\sqrt c-\delta\sqrt c\big)
\ge L\Big(1-R\sqrt c-\frac{1-R\sqrt c}{2}\Big)=\frac{1-R\sqrt c}{2}\,L>0,
\end{equation}
and $t-\rho^{2}\ge\tfrac{15}{16}t>0$. Let
$Q_{\rho}(x,t):=(t-\rho^{2},t)\times(x-\rho,x+\rho)$. By \eqref{eq:geom} the parabolic
cylinder $Q_{\rho}(x,t)\cap\big((0,L)\times(0,t)\big)$ lies above $\{s=0\}$ and its points
satisfy $x'\le R\sqrt t+\rho\le\tfrac12L$ in the small-$c$ regime.
We apply Lemma~\ref{lem:image} directly at every $(x',s)$ in this
cylinder, instead of invoking a similarity window scaled with a different time.
Decrease $c_2$, and increase $L_0$, so that
\[
 q:=\frac{R_0}{L}+(R+\tfrac14)\sqrt c\le\frac{\varepsilon}{4}
 \qquad(c\le c_2,\ L\ge L_0).
\]
Since $s\le t=cL^2$ and $x'/L\le(R+\tfrac14)\sqrt c$,
\[
 \frac{(2L-R_0-x')^2}{4s}
 \ge \frac{(2-q)^2}{4c}
 =\frac{1-q+q^2/4}{c}
 \ge\frac{1-\varepsilon/4}{c}.
\]
With the choice
$\theta=(1-\varepsilon)/(1-\varepsilon/4)$ already made in Step~1, this gives
\[
 \theta\frac{(2L-R_0-x')^2}{4s}\ge\frac{1-\varepsilon}{c}.
\]
All choices are independent of $c$ and $L$.  Hence
$\|z\|_{L^\infty(Q_\rho\cap((0,L)\times(0,t)))}
\le Ce^{-(1-\varepsilon)/c}$.
Since $z$ solves the heat equation exactly, the classical interior derivative estimates
for the heat kernel give
\begin{equation}\label{eq:zder}
|\partial_{x}^{j}z(x,t)|\le C_{j}\,\rho^{-j}\,
\|z\|_{L^{\infty}(Q_{\rho}\cap((0,L)\times(0,t)))}
\le C_{j}\,t^{-j/2}\,e^{-(1-\varepsilon)/c},
\qquad 0\le x\le R\sqrt t,
\end{equation}
where $\delta^{-j}$ has been absorbed into $C_{j}$. This is legitimate with $C_{j}$
independent of $c$: for $c\le c_{2}$ one has $R\sqrt c\le\varepsilon/8$, hence
$1-R\sqrt c\ge 1-\varepsilon/8\ge\tfrac12$ and therefore
$\delta=\tfrac14$, so that
$\delta^{-j}\le 4^{j}$ is bounded independently of $c$ and $L$; for
$c\in[c_{2},1/R^{2})$ the derivative estimate follows from
Lemma~\ref{lem:moderate-c}. For
$x$ within distance $\rho$ of $x=0$ the same bound follows by applying \eqref{eq:zder} to
the odd extension $\tilde z(x,s):=-z(-x,s)$, $x<0$, which solves the heat equation across
$x=0$ and has the same $L^{\infty}$ size, precisely because $z(0,s)=0$ for all $s$.

\emph{Step 3 (bounds on $w^{L}$, $w^{\infty}$).} By \eqref{eq:w0bounds}
the initial and boundary data of $w^{L}$ and $w^{\infty}$ lie in
$[e^{-\|u_{0}\|_{L^{1}}},e^{\|u_{0}\|_{L^{1}}}]$, so the maximum
principle gives
\begin{equation}\label{eq:wbounds2}
0<m_{0}:=e^{-\|u_{0}\|_{L^{1}}}\le w^{L},\,w^{\infty}
\le e^{\|u_{0}\|_{L^{1}}} .
\end{equation}
The derivative bounds established in the proof of
Lemma~\ref{lem:moderate-c} --- the Gaussian-kernel bound for
$w^{\infty}$ and \eqref{eq:wL-der} for $w^{L}$, the latter uniform in
$0<c<R^{-2}$ --- give
\begin{equation}\label{eq:wder}
|\partial_{x}^{m}w^{\infty}(x,t)|+|\partial_{x}^{m}w^{L}(x,t)|
\le C_{m}t^{-m/2},
\qquad 0\le x\le L,\ m\ge1,
\end{equation}
with $C_m$ independent of $c$ and $L$. Consequently, by Fa\`a di Bruno
together with \eqref{eq:wbounds2} and \eqref{eq:wder},
\begin{equation}\label{eq:invwder}
\Big|\partial_{x}^{b}\Big(\frac{1}{w^{L}}\Big)\Big|\le C_{b}t^{-b/2},
\qquad
\Big|\partial_{x}^{b}\Big(\frac{1}{w^{\infty}}\Big)\Big|\le C_{b}t^{-b/2},
\qquad b\ge0 .
\end{equation}

\emph{Step 4 (transfer to $u$).} Since $u=w_{x}/w$ and
$\frac{1}{w^{L}}-\frac{1}{w^{\infty}}=\frac{w^{\infty}-w^{L}}{w^{L}w^{\infty}}
=-\frac{z}{w^{L}w^{\infty}}$,
\begin{equation}\label{eq:key}
u^{L}-u^{\infty}
=\frac{z_{x}}{w^{L}}-\frac{w^{\infty}_{x}}{w^{L}w^{\infty}}\,z .
\end{equation}
Differentiating \eqref{eq:key} $k$ times and expanding by Leibniz, one obtains finitely
many terms, each carrying exactly one factor $\partial_{x}^{a}z$ with $0\le a\le k+1$, the
remaining $k+1-a$ derivatives being distributed among factors $w^{\infty}_{x}$, $1/w^{L}$
and $1/w^{\infty}$. By \eqref{eq:zder}, \eqref{eq:wder} and \eqref{eq:invwder} each such
term is bounded by
\[
C\,t^{-a/2}e^{-(1-\varepsilon)/c}\cdot t^{-(k+1-a)/2}
=C\,t^{-(k+1)/2}\,e^{-(1-\varepsilon)/c},
\]
the homogeneities adding to $a+(k+1-a)=k+1$. Summing,
\begin{equation}\label{eq:udiff}
|\partial_{x}^{k}(u^{L}-u^{\infty})(x,t)|
\le C_{k}\,t^{-(k+1)/2}\,e^{-(1-\varepsilon)/c},
\qquad 0\le x\le R\sqrt t .
\end{equation}
Since $\partial_{\xi}^{k}v=t^{(k+1)/2}\partial_{x}^{k}u$ by \eqref{eq:rescaled}, this is
\eqref{ck}.

It remains to justify the sharpness assertion, since
Theorem~\ref{thm:sharp}(ii) is stated for the limiting profiles.  Fix $c>0$ and
take $L\to\infty$.  Theorem~\ref{thm:Phi} gives
$v^L(0,cL^2)\to\Phi_c(0)$, while
Theorem~\ref{thm:halflinerate} gives
$v^\infty(0,cL^2)\to f_M(0)$.  Thus
\[
 \lim_{L\to\infty}|v^L(0,cL^2)-v^\infty(0,cL^2)|
 =|\Phi_c(0)-f_M(0)|.
\]
If a uniform estimate held with exponent $\gamma>1$, passage to this limit
would give $|\Phi_c(0)-f_M(0)|\le Ce^{-\gamma/c}$, contradicting
Theorem~\ref{thm:sharp}(ii) as $c\downarrow0$.
\end{proof}

\subsection{Two-scale asymptotics and comparison of regimes}\label{ss:regimes}

By Theorem \ref{thm:halflinerate}, for data satisfying \eqref{eq:hyp},
\begin{equation}\label{eq:hlrate}
\|v^{\infty}(\cdot,t)-f_{M}\|_{L^{\infty}(0,R)}\le C\,t^{-1/2},
\qquad t\ge1,
\end{equation}
with $C=C(\|u_{0}\|_{L^{1}},R_{0},R)$.

\begin{cor}\label{cp}
Assume \eqref{eq:hyp} and fix $R>0$, $p\in[1,\infty]$ and $\varepsilon\in(0,1)$. There is
$C=C(\|u_{0}\|_{L^{1}},R_{0},R,p,\varepsilon)$ such that, for all $L\ge L_{0}$ and all
$t\ge1$ with $c=t/L^{2}<1/R^{2}$,
\begin{equation}\label{ep}
\big\|v^{L}(\cdot,t)-f_{M}\big\|_{L^{p}(0,R)}
\le
C\,e^{-(1-\varepsilon)L^{2}/t}+C\,t^{-1/2} .
\end{equation}
In particular $v^{L}(\cdot,t)\to f_{M}$ in $L^{p}(0,R)$ along any sequence with
$t\to\infty$ and $t/L^{2}\to0$, that is throughout the window $1\ll t\ll L^{2}$; the two
terms balance at $t\asymp L^{2}/\ln L$, where the error is $O\big((\ln L)^{1/2}/L\big)$.
\end{cor}

\begin{proof}
By the triangle inequality,
\[
\|v^{L}-f_{M}\|_{L^{p}(0,R)}
\le\|v^{L}-v^{\infty}\|_{L^{p}(0,R)}+\|v^{\infty}-f_{M}\|_{L^{p}(0,R)} .
\]
Theorem \ref{thm:main-convergence} with $k=0$ bounds the first term by
$R^{1/p}C_{0}e^{-(1-\varepsilon)/c}$, and \eqref{eq:hlrate} bounds the second by
$R^{1/p}Ct^{-1/2}$. Substituting $c=t/L^{2}$ gives \eqref{ep}, and the balance follows from
$e^{-(1-\varepsilon)L^{2}/t}=t^{-1/2}$.
\end{proof}

Estimate \eqref{ep} is a fixed-$(L,t)$ inequality whose two terms have distinct origins:
the first is the exponentially small cost of replacing the half-line by the interval at
the diffusive scale, the second the polynomial relaxation of the half-line dynamics toward
the universal profile. Together they exhibit the window $1\ll t\ll L^{2}$ in which $v^{L}$
is close to $f_{M}$.

\begin{remark}\label{rem:regimes}
Fix $R>0$ and write $c=t/L^{2}$.
\begin{enumerate}
\item \emph{Self-similar regime, $1\ll t\ll L^{2}$.} By Theorem \ref{thm:main-convergence},
$\|v^{L}(\cdot,t)-v^{\infty}(\cdot,t)\|_{C^{k}([0,R])}\le C_{k}e^{-(1-\varepsilon)L^{2}/t}\to0$,
while $v^{\infty}(\cdot,t)\to f_{M}$ in $C^{k}_{\mathrm{loc}}$ by Theorem
\ref{thm:halflineCk}; hence $v^{L}(\cdot,t)\to f_{M}$. Corollary \ref{cp} quantifies the
window, and Theorem \ref{thm:sharp} shows the rate $e^{-L^{2}/t}$ is sharp.
\item \emph{Transition regime, $t\sim L^{2}$.} For each fixed $c>0$, Theorem \ref{thm:Phi}
gives $v^{L}(\cdot,cL^{2})\to\Phi_{c}$ in $C^{k}_{\mathrm{loc}}([0,c^{-1/2}))$, with
$\Phi_{c}$ explicit. The family tends to $f_M$ in similarity coordinates as $c\to0$
and to $\mathcal U_M$ only after the domain-scale rescaling as $c\to\infty$
(Corollary \ref{cor:c-to-zero} and Proposition \ref{prop:cinfty}), and its existence identifies
$t\sim L^{2}$ as the critical scale.
\item \emph{Stationary regime, $t\gg L^{2}$.} By Theorem \ref{expt1},
$\|u^{L}(\cdot,t)-U_{M}^{L}\|_{L^{\infty}(0,L)}\le C(L)\,e^{-\pi^{2}t/L^{2}}$ for $t\ge t_{0}$,
consistently with Proposition \ref{prop:cinfty}.
\end{enumerate}
The three regimes are encoded by the single family $\Phi_c$, with the qualification
that its $c\to\infty$ limit must be taken after domain-scale rescaling.
\end{remark}

\begin{remark}\label{rem:notmetastable}
The transient described here lasts $t\sim L^{2}$, the diffusive time of a domain of
diameter $L$: it is the natural time scale of the problem rather than an anomalously long
one, and the solution follows the half-line dynamics for exactly as long as the boundary
is out of diffusive reach. This is the principle of not feeling the boundary
\cite{Kac1951}, and the accurate description is Barenblatt's intermediate asymptotics
\cite{barenblatt1996}. It is distinct from the metastability of \cite{KT01,BW09},
where the small parameter is the viscosity and the transient exceeds the natural scale,
and from the exponentially slow, boundary-driven shock motion for Burgers on an interval
studied in \cite{ReynaWard1995,Kreiss1986,MasciaStrani2013}.
\end{remark}

\begin{remark}
The compact support assumption \eqref{eq:hyp} is made for simplicity. The results extend to
data with exponential decay, $|u_{0}(x)|\le Ce^{-\alpha x}$: the Gaussian tail estimate of
Lemma \ref{lem:boundary-tail} is replaced by the corresponding heat-kernel estimate
against an exponentially decaying datum, the layer in Lemma \ref{lem:hatw} has width
$O(\alpha^{-1}\log L/L)$, and the exponent $\gamma=1$ is unchanged provided
$\alpha^{-1}=o(L)$. For data with polynomial decay the exponent is dictated by the tail
rather than by the image charge.
\end{remark}

\section{\texorpdfstring{Robust interval conclusions for space-dependent diffusion}{Robust interval conclusions for space-dependent diffusion}}\label{sec:variable}

The explicit analysis of the previous sections rests on the constant-coefficient
heat kernel. In many models the diffusion is heterogeneous, and it is natural to
ask which of the phenomena described above --- the two attractors, the relaxation
rate, the diffusive transition --- survive when the diffusivity varies in space.
We consider
\begin{equation}\label{eq:varflux}
u_t=\partial_x\big(a(x)\,(u_x+u^2)\big),
\qquad (x,t)\in(0,L)\times(0,\infty),
\end{equation}
with the conservative boundary conditions
\begin{equation}\label{eq:varbc}
a(x)\,(u_x+u^2)=0 \quad\text{at } x=0,L,
\end{equation}
which, since $a>0$, read again $u_x+u^2=0$ at the endpoints and make the flux
vanish there, so that the mass $M=\int_0^L u\,dx$ is conserved. Throughout this
section $a$ is smooth and uniformly elliptic,
\begin{equation}\label{eq:aelliptic}
0<a_0\le a(x)\le a_1<\infty .
\end{equation}
This is the only assumption used for the interval results below, which therefore
hold for any diffusivity satisfying \eqref{eq:aelliptic}, including one varying on
the scale of the domain. The half-line discussion at the end of the section
requires in addition that $a$ settle to a constant at infinity, in the spirit of
Duro and Zuazua \cite{DuroZuazua1999},
\begin{equation}\label{eq:ainfty}
|a(x)-a_\infty|\le C(1+x)^{-\delta},
\qquad
|a_x(x)|\le C(1+x)^{-\delta-1},
\qquad a_\infty>0,\ \delta>0;
\end{equation}
it is introduced there and not before.

The flux in \eqref{eq:varflux} is the one compatible with the Hopf--Cole
transformation: setting $v=\int_0^x u$ and $w=e^v$, the boundary conditions
become $w(0,t)=1$, $w(L,t)=e^M$ as in \eqref{ebw}, and a direct computation gives
\begin{equation}\label{eq:varheat}
w_t=a(x)\,w_{xx},
\qquad (x,t)\in(0,L)\times(0,\infty),
\end{equation}
the heat equation with variable diffusivity, in non-divergence form, with the
same constant Dirichlet data as before. Since \eqref{eq:varheat} is uniformly
parabolic with smooth coefficient, for continuous positive $w_0$ it has a unique
classical solution, positive by the maximum principle, and $u=w_x/w$ is the
unique solution of \eqref{eq:varflux}--\eqref{eq:varbc} in the class of
Section~\ref{FI}, with $\int_0^L u\,dx=M$ for all $t>0$; the proof of
Theorem~\ref{T1} applies verbatim, the coefficient $a$ entering only through the
linear theory.


A stationary solution of \eqref{eq:varflux} has constant flux,
$a(x)\,(U_x+U^2)\equiv\kappa$. Evaluating \eqref{eq:varbc} at $x=0$ gives
$\kappa=0$, so $a(x)(U_x+U^2)\equiv0$ throughout and, since $a>0$,
\[
U_x+U^2=0 \quad\text{on all of }(0,L),
\]
independently of $a$. The unique stationary state of mass $M$ is therefore again
\eqref{UM}. This propagation from the boundary to the interior is special to the
stationary problem, where the flux is constant in $x$; the time-dependent
solution does feel $a$ in the interior. Thus the geometry of the diffusion
affects the approach to equilibrium, but not the equilibrium itself.

\subsection{Relaxation rate: the principal eigenvalue}

The rate of approach to $U_M^L$ is set by the spectral gap of the operator
$a(x)\partial_{xx}$ with Dirichlet conditions. It is symmetric in the weighted space
$L^2\big((0,L),a^{-1}dx\big)$, the weight cancelling the coefficient so that two
integrations by parts move $\partial_{xx}$ from $\phi$ to $\psi$, the boundary
terms vanishing by the Dirichlet condition. Its principal eigenvalue
$\lambda_1(a)$, the least $\lambda$ for which $a\phi''+\lambda\phi=0$,
$\phi(0)=\phi(L)=0$ admits a nontrivial solution, has the variational
characterization
\begin{equation}\label{eq:rayleigh}
\lambda_1(a)
=\inf_{\substack{\phi\in H^1_0(0,L)\\ \phi\ne0}}
\frac{\displaystyle\int_0^L (\phi_x)^2\,dx}
     {\displaystyle\int_0^L \phi^2\,a^{-1}\,dx}\,.
\end{equation}
Since the numerator does not involve $a$ and the denominator is monotone in
$a^{-1}$, \eqref{eq:aelliptic} yields
\begin{equation}\label{eq:ev-bounds}
\frac{\pi^2}{L^2}\,a_0\ \le\ \lambda_1(a)\ \le\ \frac{\pi^2}{L^2}\,a_1,
\end{equation}
with equality when $a$ is constant, recovering $\lambda_1=\pi^2/L^2$ of
Theorem~\ref{expt1}.

\begin{teo}\label{thm:var-relax}
Let $u_0\in L^1(0,L)$ have mass $M$ and let $u$ solve
\eqref{eq:varflux}--\eqref{eq:varbc}. For every $p\in[1,\infty]$ there is
$C=C(\|u_0\|_{L^1},a,L,p)>0$ such that
\begin{equation}\label{eq:var-relax}
\|u(\cdot,t)-U_M^L\|_{L^p(0,L)}
\le C\,\bigl(1+t^{-3/4}\bigr)\,e^{-\lambda_1(a)\,t},
\qquad t>0 .
\end{equation}
\end{teo}

\begin{proof}
The proof of Theorem~\ref{expt1} carries over once $\pi^2/L^2$ is replaced by
$\lambda_1(a)$. With $z:=w-W$, $W(x)=1+(e^M-1)x/L$, expanding $z(\cdot,0)$ in
the eigenbasis of $a\partial_{xx}$ gives decay at the rate $e^{-\lambda_1(a)t}$
in $L^2(a^{-1}dx)$, a norm equivalent to $L^2$ by \eqref{eq:aelliptic}. Since
\eqref{eq:varheat} is uniformly parabolic, the parabolic estimates give
\eqref{eq:zsmoothing} with $\lambda_1(a)$ in place of $\lambda_1$ and with $C$
depending also on $a_0,a_1,\|a\|_{C^1}$, the two time regimes being separated as
there. The transfer to $u=w_x/w$ is unchanged.
\end{proof}

The bounds \eqref{eq:ev-bounds} show that the diffusivity enters the relaxation
rate only through $\lambda_1(a)$, squeezed between the rates of the constant
diffusivities $a_0$ and $a_1$; a slower diffusion anywhere in the interval lowers
the gap, the effect being global, mediated by the principal eigenfunction, rather
than localized where $a$ is small.

\subsection{\texorpdfstring{Limits of the present argument on the half-line}{Limits of the present argument on the half-line}}

On the half-line the relevant regime is self-similar, and there the diffusivity
must be compared with a constant one; we assume \eqref{eq:ainfty}. For the constant
coefficient $a_\infty$, the self-similar profile of mass $M$ is obtained from $f_M$
by pure scaling: substituting $u=t^{-1/2}F(x/\sqrt t)$ in
$u_t=a_\infty(u_{xx}+(u^2)_x)$ gives
$-\frac12F-\frac12\xi F'=a_\infty(F''+(F^2)')$, solved by
\begin{equation}\label{eq:fMainfty}
f_M^{(a_\infty)}(\xi)
=\frac{1}{\sqrt{a_\infty}}\,f_M\!\Big(\frac{\xi}{\sqrt{a_\infty}}\Big),
\qquad \int_0^\infty f_M^{(a_\infty)}=M .
\end{equation}

For the variable coefficient the picture is best seen through $h:=w_x$, which by
\eqref{eq:varheat} solves the divergence-form equation
\begin{equation}\label{eq:hdiv}
h_t=(a(x)\,h_x)_x \qquad\text{on }(0,\infty),
\end{equation}
with the no-flux condition inherited at $x=0$ and conserved mass
$\int_0^\infty h(\cdot,0)=e^M-1$, exactly as in the constant-coefficient analysis
of Theorem~\ref{thm:halflinerate}. Under \eqref{eq:ainfty} the coefficient of
\eqref{eq:hdiv} settles to $a_\infty$, and \eqref{eq:hdiv} is then a
divergence-form diffusion with asymptotically constant coefficient, the setting of
Duro and Zuazua \cite{DuroZuazua1999} for the analogous whole-space problem.
This analogy suggests convergence toward $f_M^{(a_\infty)}$, but neither a
half-line convergence theorem nor a rate is needed for the results of this
paper.  We therefore do not formulate an unproved rate as part of the main
claims.

The boundary places \eqref{eq:hdiv} outside the whole-space statements of
\cite{DuroZuazua1999}: a proof would combine their scaling analysis with Gaussian
bounds for the Neumann heat kernel of the divergence-form operator
$\partial_x(a\,\partial_x)$ \cite{Aronson1968}, the point requiring care being the
uniformity of these bounds near the origin, where $a$ departs from $a_\infty$. For
slowly decaying coefficients ($\delta$ small) the tail of $a$ may itself limit the
rate. We leave this separate problem open.

The confinement mechanism does not survive in explicit form. The change of
variable $z(x)=\int_0^x ds/\sqrt{a(s)}$ turns \eqref{eq:varheat} into a heat
equation with a first-order drift $-\tfrac12(a_x/\sqrt a)\,W_z$ that is integrable
in $z$ under \eqref{eq:ainfty}, so that in these coordinates the interval $(0,L)$
has effective length $d(0,L)=\int_0^L ds/\sqrt{a(s)}$ and the Gaussian cost of
reaching the boundary and returning is measured in the metric $ds/\sqrt a$. But the
family $\Phi_c$, the sharp exponent and the matching lower bound of
Theorem~\ref{thm:sharp} all rest on the closed-form image representation of
Lemma~\ref{lem:poisson}, which is unavailable for variable $a$. The determination
of the sharp confinement constant, presumably a functional of the metric
$ds/\sqrt a$, is open.

Finally, the reduction above tames the coefficient only because \eqref{eq:ainfty}
forces $a$ to settle at infinity. The genuinely different case is a diffusivity
varying on the scale of the domain, $a(x)=\alpha(x/L)$ for a fixed profile
$\alpha$: the rescaled coefficient does not converge as $L\to\infty$, the
similarity structure is lost, and no self-similar profile is available. The
interval results of this section  (well-posedness, mass conservation, the
stationary state \eqref{UM}, and the relaxation rate \eqref{eq:var-relax})
already cover this regime, since they use only \eqref{eq:aelliptic}. What is open
is the transition: whether a family analogous to $\Phi_c$ exists, and whether the
confinement exponent is then a genuine functional of $\alpha$.

\section{Numerical illustration}\label{NS}

We illustrate the three regimes numerically, always with mass $M=1$ and initial
datum $u_0(x)=2(1-x)\mathbf 1_{[0,1]}(x)$, from which
$w_0(x)=\exp(\int_0^x u_0)$. The observation window is stated in each caption.

The simulations use the Hopf--Cole transformation to avoid discretizing the
nonlinear boundary conditions: we solve the linear heat equation for $w$ by an
explicit finite difference scheme and recover $u=w_x/w$. Since the data
$w(0,t)=1$, $w(L,t)=e^M$ are imposed exactly, the mass
$M=\ln w(L,t)-\ln w(0,t)$ is exact at the continuous level, and at the discrete
level it is conserved up to the $O(\Delta x^2+\Delta t)$ error of the
reconstruction of $u$. In higher dimensions, the absence of such a
transformation would require treating the nonlinear boundary condition
directly, which typically introduces mass loss.

For reproducibility: the same mesh $x_j=j\Delta x$ is used for every $L$, with
$\Delta x=0.02$ and $\mu=\Delta t/\Delta x^2=0.4$, within the stability
constraint $\mu\le1/2$; interior values are updated by
$w_j^{n+1}=w_j^n+\mu(w_{j+1}^n-2w_j^n+w_{j-1}^n)$ with $w_0^n=1$, $w_N^n=e^M$;
$w_x$ is approximated by second-order centered differences at interior nodes and
second-order one-sided differences at the endpoints; discrete supremum norms are
maxima over nodes in the stated window and discrete $L^p$ norms use composite
trapezoidal quadrature. Both $U_M^L$ and $f_M$ are known explicitly and serve as
references. Figure~\ref{fig4f2} uses the closed formulas and no half-line
truncation.

As a grid-refinement check, Table~\ref{tab:refinement} compares the
finite-difference solution at $L=20$, $t=0.5$ with the sine-series
representation of $w-W$ (500 modes), the supremum error being computed for
$u=w_x/w$ on the full interval and $\Delta t$ being in each run the largest
subdivision of $t$ not exceeding $0.4\Delta x^2$. The observed orders are close
to two, as expected. Raising the reference truncation to 750 and 1000 modes
changes the reference values by less than $10^{-15}$ at all nodes used.
\begin{table}[htbp]
\centering
\begin{tabular}{c c c c}
$\Delta x$ & $\Delta t$ & $\|u_{\Delta x}-u_{\rm ref}\|_{\infty}$ & observed order\\
\hline
$0.04$ & $6.3939\times10^{-4}$ & $4.4952\times10^{-4}$ & --\\
$0.02$ & $1.6000\times10^{-4}$ & $1.1666\times10^{-4}$ & $1.95$\\
$0.01$ & $4.0000\times10^{-5}$ & $2.9697\times10^{-5}$ & $1.97$
\end{tabular}
\caption{Grid refinement for the finite-difference reconstruction of $u$ at
$L=20$ and $t=0.5$.}
\label{tab:refinement}
\end{table}
\par\medskip

We first illustrate the sharp rate of Subsection~\ref{ss:sharp}.
\begin{figure}[htbp]
\centering
\def\svgwidth{1\linewidth}
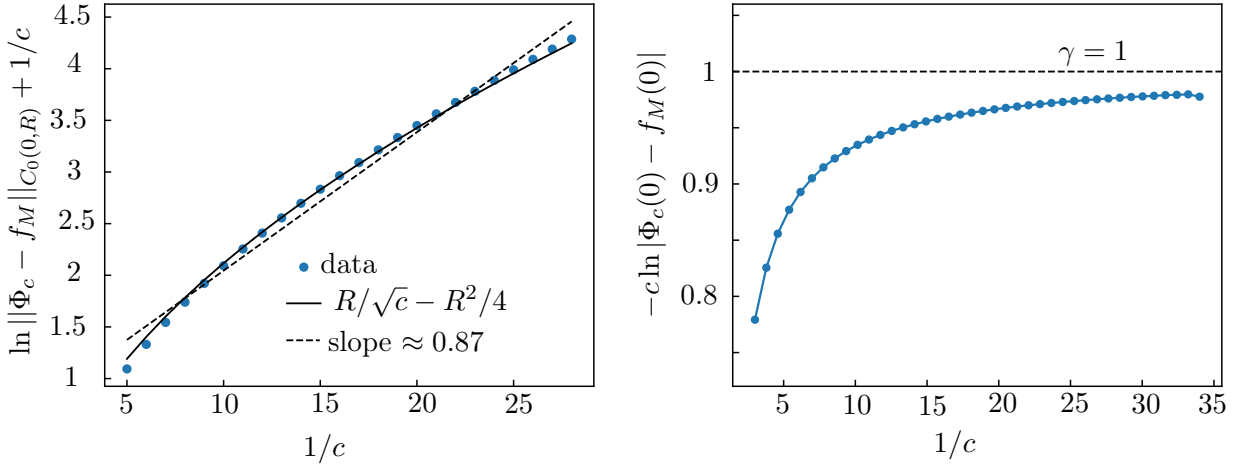
\caption{The sharp rate $\gamma=1$ at $M=1$, $R=1$, computed from the closed forms
of $\Phi_c$ and $f_M$.}
\label{fig4f2}
\end{figure}

Since $\Phi_c$ and $f_M$ are known in closed form, Figure~\ref{fig4f2} is
computed to machine precision, with no discretization; as $\Phi_c$ is the
$L\to\infty$ limit of $v^L(\cdot,cL^2)$ (Theorem~\ref{thm:Phi}), it displays the
transition in the limit that isolates the exponent. The right panel is decisive:
$-c\ln|\Phi_c(0)-f_M(0)|$, which Theorem~\ref{thm:sharp}(ii) asserts converges to
$1$, reaches $0.98$ by $1/c=34$. The left panel shows why a straight-line fit of
$\ln E$ against $1/c$ misleads: it returns a slope near $0.87$ that drifts with
the window, because the leading image contribution and the upper bound carry the
factor $e^{-1/c+R/\sqrt c-R^2/4}$, whose subexponential part $e^{R/\sqrt c}$
flattens the apparent slope. Removing it leaves $0.995$, in agreement with
$\gamma=1$ (Remark~\ref{rem:fit}).

The next simulation illustrates Corollary~\ref{cp}, where $c=t/L^2$ in
\eqref{ep}. Letting $c\to0$ is not by itself enough: the estimate is informative
only inside the window $1\ll t\ll L^2$, so $L$ and $t$ must grow together, with
$t=cL^2\to\infty$ and $c\to0$. Figure~\ref{fig2s} compares finite values of $L$
and $t$ and is a qualitative illustration of that separation of scales rather
than a test of the limit itself, since the smallest displayed times do not
satisfy $1\ll t$. It nevertheless shows the predicted ordering: profiles with
smaller $t/L^2$ stay closer to the half-line self-similar profile, and the
shorter intervals feel the remote boundary first.

\begin{figure}[htbp]
\centering
\def\svgwidth{1\linewidth}
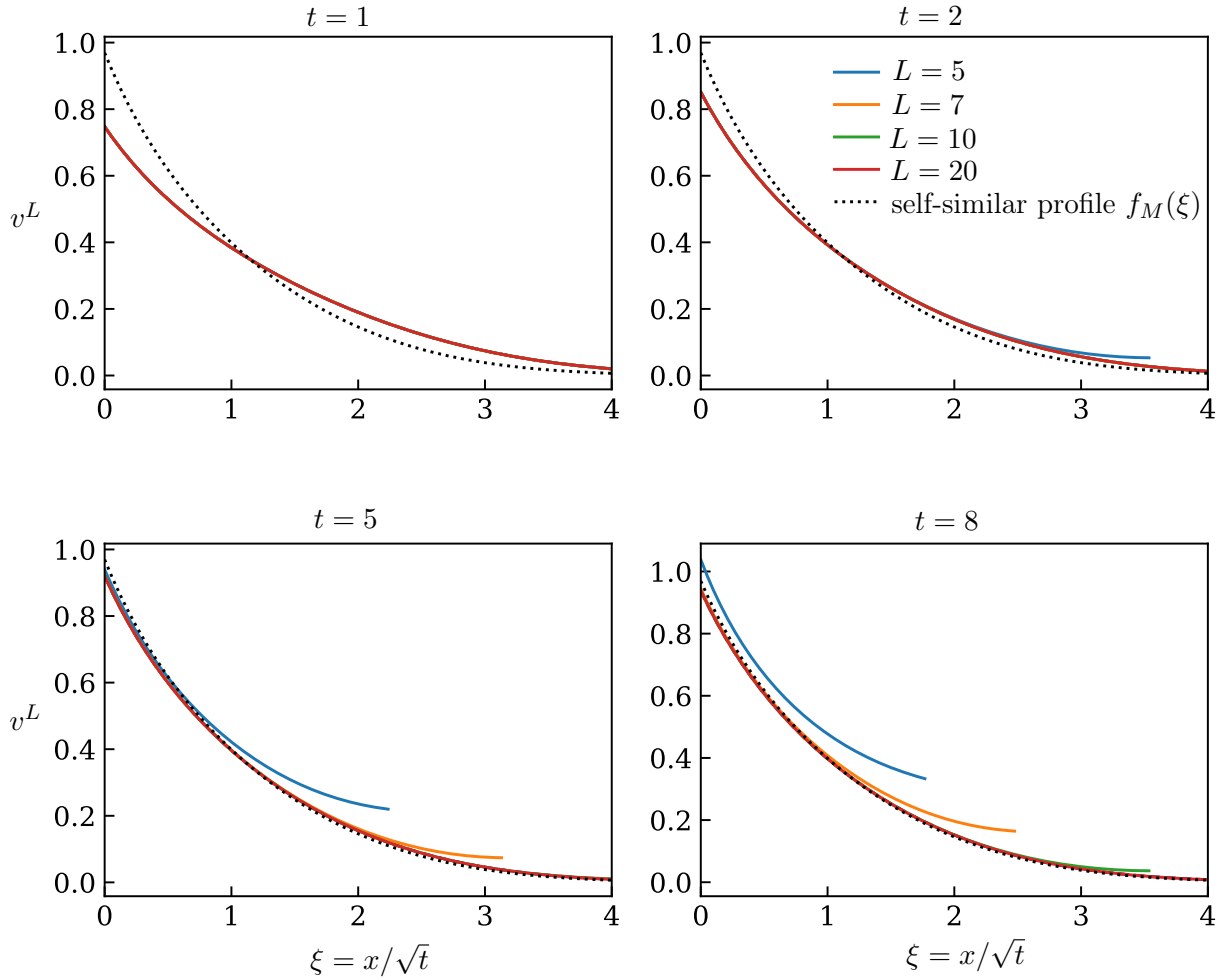
\caption{Profiles of $v^L(\xi,t)=\sqrt{t}\,u^L(\xi\sqrt{t},t)$ in similarity variables $\xi=x/\sqrt{t}$ for
$L\in\{5,7,10,20\}$ at times $t\in\{1,2,5,8\}$.}
\label{fig2s}
\end{figure}

At $t=1$ the four profiles are clustered, still under the influence of the
initial datum; at $t=2$ they have moved closer to $f_M$, and $v^5$ already
detaches, being the first to feel the boundary (its similarity interval is then
$0<\xi<5/\sqrt2\approx3.5$). At $t=5$ the deviation of $v^5$ is pronounced and
$v^7$ starts to follow, while $v^{10}$ and $v^{20}$ remain very close to $f_M$;
at $t=8$ so does $v^{10}$, whereas $v^{20}$ is still clustered with $f_M$.

Fixing $L=20$, we now let $t$ increase, expecting $u^L$ to approach the
self-similar profile first and the interval equilibrium afterwards. To make the
second convergence visible we return to the physical variable $x$, in which the
self-similar profile reads
$\tilde f_M(x,t)=t^{-1/2}f_M(x/\sqrt t)$.

\begin{figure}[htbp]
\centering
\def\svgwidth{1\linewidth}
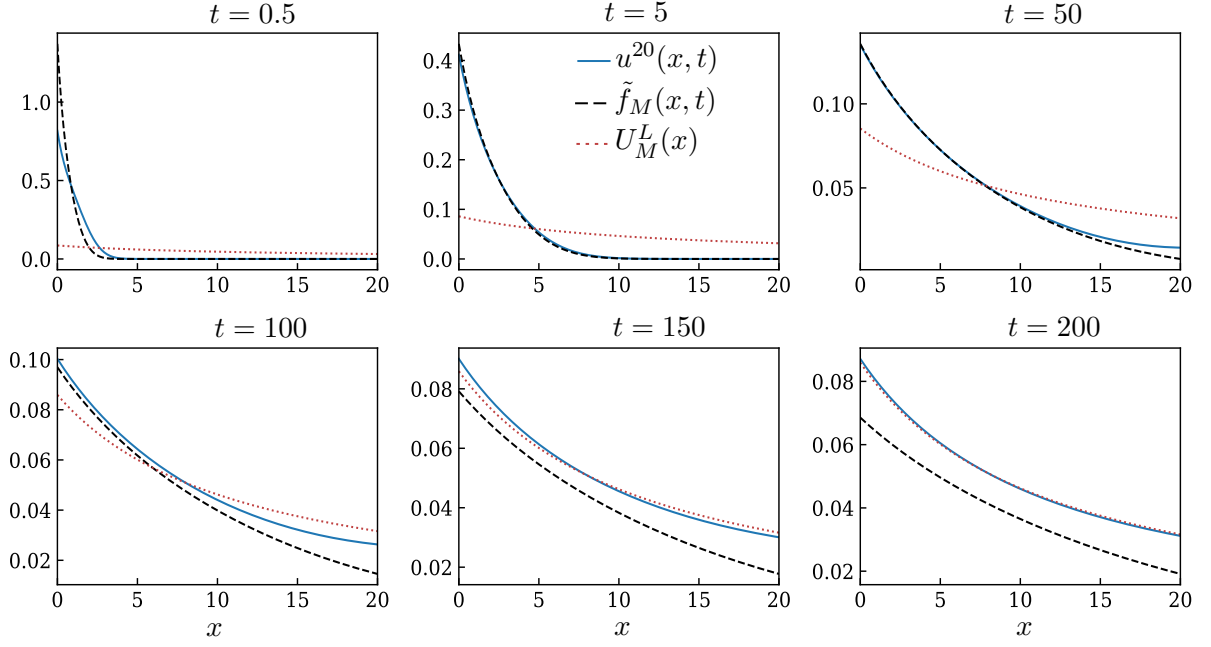
\caption{Evolution of $u^{20}$ ($u^L$, for $L=20$) from self-similar dynamics ($\tilde{f}_M$) to the stationary regime ($U_M^L$).}
\label{fig3b}
\end{figure}

At $t=0.5$ the profile is still dominated by the initial datum; by $t=5$ it is
very close to $\tilde f_M$, and it departs visibly from that diffusive regime
around $t=50$. The panel $t=100$ shows the transition phase between the two
regimes; by $t=150$ the solution is already near $U_M^L$, and at $t=200$ the two
curves are indistinguishable.

\section{\texorpdfstring{Conclusions and outlook}{Conclusions and outlook}}\label{multi}

We have resolved the transition between half-line and bounded-domain
asymptotics for the viscous Burgers equation with conservative boundary
conditions. At the critical scale $t=cL^2$, the rescaled solution converges to
an explicit family $\Phi_c$ connecting the nonlinear self-similar profile to
the nonconstant interval equilibrium. Before confinement dominates, the
remote-boundary correction is exponentially small, with sharp order
$e^{-L^2/t}$ and optimal constant $1$. The result therefore provides the four
elements of a complete crossover description: the relevant scale, the profile
governing the transition, its two limiting states, and the sharp rate at which
the boundary becomes visible.

The interval/half-line pair is the one-dimensional instance of a broader
truncation question: whether dynamics on bounded domains $\Omega_L$ shadows
that on an unbounded limit domain up to time $t\sim L^2$.  Cones are a natural
next setting because they preserve diffusive scaling.  The same one-dimensional
argument already extends to $(-L,L)$, with the whole line as limiting domain
\cite{evz1,ez}.

For a multidimensional viscous conservation law
\[
 u_t=\nabla\!\cdot(\nabla u+F(u)),\qquad
 (\nabla u+F(u))\cdot n=0,
\]
mass is still conserved, but the two tools responsible for the exact result
above disappear: there is generally no Hopf--Cole transformation, and heat
propagation in a cone is governed by the angular spectrum rather than a
one-dimensional image expansion.  Establishing a nonlinear crossover, and
identifying the geometric constant replacing $1$, therefore require methods
beyond those of this paper.

\section*{Acknowledgments}

M. Sonego has been partially supported by the Conselho Nacional de Desenvolvimento Científico e Tecnológico (CNPq), Grant/Award Number: 311893/2022-8; Fundação de Amparo à Pesquisa do Estado de Minas Gerais (FAPEMIG), Grant/Award Number: RED-00133-21 and APQ-02060-25.

E. Zuazua was funded by the Alexander von Humboldt Professorship program, the
ERC Advanced Grant CoDeFeL, Grant PID2023-146872OB-I00-DyCMaMod of MICIU
(Spain), COST Actions CA24136 (InterCoML) and CA24122 (mSPACE), supported by
COST (European Cooperation in Science and Technology), project AFOSR 24IOE027,
and SURE-AI Centre grant 357482 of the Research Council of Norway.

\section*{Data and code availability}

No external data were used in this study.  The numerical results are
generated from the initial datum, closed formulas, and finite-difference scheme
specified in Section~\ref{NS}.  The script
\texttt{numerical\_refinement\_v12.py}, supplied as a supplementary file,
reproduces Table~\ref{tab:refinement}.  The figure-generation code and the
resulting numerical arrays are available from the corresponding author upon
reasonable request and will be deposited in a public repository upon
acceptance.

\section*{Declaration of competing interest}

The authors declare that they have no known competing financial interests
or personal relationships that could have appeared to influence the work
reported in this paper.

\section*{Declaration of generative AI and AI-assisted technologies in the writing process}
During the preparation of this work the authors used ChatGPT (GPT-5, OpenAI) and Claude (Opus 5, Anthropic) to assist with drafting, restructuring, and proofreading the text, as well as generating and testing the illustrative scripts. All references were checked against the original sources. After using these tools, the authors reviewed and edited the content as needed and take full responsibility for the content of the publication.


\begin{thebibliography}{99}

\bibitem{Aronson1968}
D. G. Aronson,
\textit{Non-negative solutions of linear parabolic equations},
Ann. Scuola Norm. Sup. Pisa, Ser.~III, 22 (1968), no.~4, 607--694.

\bibitem{barenblatt1996}
G. I. Barenblatt,
\textit{Scaling, Self-similarity, and Intermediate Asymptotics},
Cambridge University Press, Cambridge, 1996.

\bibitem{BarenblattZeldovich1972}
G. I. Barenblatt and Ya. B. Zel'dovich,
\textit{Self-similar solutions as intermediate asymptotics},
Ann. Rev. Fluid Mech., 4:285--312, 1972.

\bibitem{BW09}
M. Beck and C. E. Wayne,
\textit{Using global invariant manifolds to understand metastability in the Burgers equation with small viscosity},
SIAM J. Appl. Dyn. Syst., 8(3):1043--1065, 2009.

\bibitem{BertiniPonsiglione2012}
L. Bertini and M. Ponsiglione,
\textit{A variational approach to the stationary solutions of the Burgers equation},
SIAM J. Math. Anal., 44(2):682--698, 2012.

\bibitem{Karch2000}
P. Biler and G. Karch,
\textit{A Neumann problem for a convection-diffusion equation on the half-line},
Ann. Polon. Math., 74:79--95, 2000.

\bibitem{Ciesielski1966}
Z. Ciesielski,
\textit{Heat conduction and the principle of not feeling the boundary},
Bull. Acad. Polon. Sci. S\'er. Sci. Math. Astronom. Phys., 14:435--440, 1966.

\bibitem{cole}
J. D. Cole,
\newblock On a quasi-linear parabolic equation occurring in aerodynamics,
\newblock {\em Quart. Appl. Math.}, 9 (1951), 225--236.




\bibitem{Davies1989}
E. B. Davies,
\textit{Heat Kernels and Spectral Theory},
Cambridge Tracts in Mathematics, vol.~92,
Cambridge University Press, Cambridge, 1989.

\bibitem{DuroZuazua1999}
G. Duro and E. Zuazua,
\textit{Large time behavior for convection-diffusion equations in $\mathbb{R}^n$ with asymptotically constant diffusion},
Commun. Partial Differential Equations, 24(7--8):1283--1340, 1999.

\bibitem{evz1}
M. Escobedo, J. L. Vázquez, and E. Zuazua,
\textit{Asymptotic behaviour and source-type solutions for a diffusion-convection equation},
Arch. Rational Mech. Anal., 124(1):43--65, 1993.

\bibitem{evz2}
M. Escobedo, J. L. Vázquez, and E. Zuazua,
\textit{Entropy solutions for diffusion-convection equations with partial diffusivity},
Trans. Amer. Math. Soc., 343(2):829--842, 1994.

\bibitem{ez}
M. Escobedo and E. Zuazua,
\textit{Long time behavior for convection-diffusion equations},
SIAM J. Math. Anal., 28(3):570--594, 1997.

\bibitem{ev}
L. C. Evans,
\textit{Partial Differential Equations},
2nd ed., Graduate Studies in Mathematics, vol.~19,
American Mathematical Society, Providence, RI, 2010.

\bibitem{Friedman1959}
A. Friedman,
\textit{Convergence of solutions of parabolic equations to a steady state},
J. Math. Mech., 8:57--76, 1959.

\bibitem{Friedman1961}
A. Friedman,
\textit{Asymptotic behavior of solutions of parabolic equations of any order},
Acta Math., 106:1--43, 1961.

\bibitem{Gushchin1984}
A. K. Gushchin,
\textit{On the uniform stabilization of solutions of the second mixed problem for a parabolic equation},
Math. USSR-Sb., 47(2):439--498, 1984.

\bibitem{hopf}
E. Hopf,
\newblock The partial differential equation $u_t + uu_x = \mu u_{xx}$,
\newblock {\em Comm. Pure Appl. Math.}, 3 (1950), 201--230.

\bibitem{Ilin1985}
A. M. Il'in,
\textit{A sufficient condition for the stabilization of the solution of a parabolic
equation}, Math. Notes, 37 (1985), no.~6, 466--469.

\bibitem{Kac1951}
M. Kac,
\textit{On some connections between probability theory and differential and integral equations},
in: Proceedings of the Second Berkeley Symposium on Mathematical Statistics and Probability,
University of California Press, Berkeley--Los Angeles, 1951, pp. 189--215.


\bibitem{KT01}
Y. J. Kim and A. E. Tzavaras,
\textit{Diffusive N-waves and metastability in the Burgers equation},
SIAM J. Math. Anal., 33(3):607--633, 2001.

\bibitem{Kreiss1986}
G. Kreiss and H.-O. Kreiss,
\textit{Convergence to steady state of solutions of Burgers' equation},
Appl. Numer. Math., 2:161--179, 1986.

\bibitem{Lieberman1996}
G. M. Lieberman,
\textit{Second Order Parabolic Differential Equations},
World Scientific, Singapore, 1996.

\bibitem{MasciaStrani2013}
C. Mascia and M. Strani,
\textit{Metastability for nonlinear parabolic equations with application to scalar viscous conservation laws},
SIAM J. Math. Anal., 45(5):3084--3113, 2013.

\bibitem{Mukminov1980}
F. Kh. Mukminov,
\textit{Stabilization of solutions of the first mixed problem for a parabolic equation of second order},
Math. USSR-Sb., 39(4):449--467, 1981.

\bibitem{ReynaWard1995}
L. G. Reyna and M. J. Ward,
\textit{On the exponentially slow motion of a viscous shock},
Comm. Pure Appl. Math., 48(2):79--120, 1995.

\bibitem{Ushakov1980}
V. I. Ushakov,
\textit{Stabilization of solutions of the third mixed problem for a second-order parabolic equation in a non-cylindrical domain},
Math. USSR-Sb., 39(1):87--105, 1981.

\bibitem{vandenBerg1989}
M. van den Berg,
\textit{Heat equation and the principle of not feeling the boundary},
Proc. Roy. Soc. Edinburgh Sect. A, 112:257--262, 1989.

\bibitem{Watanabe2016}
S. Watanabe, S. Matsumoto, T. Higurashi, and N. Ono,
\textit{Burgers equation with no-flux boundary conditions and its application for complete fluid separation},
Physica D: Nonlinear Phenomena, 331:1--12, 2016.

\end{thebibliography}
\end{document}